\documentclass[11pt,a4paper]{amsart}
\usepackage{hyperref}
\usepackage{amsfonts}
\usepackage{amsthm}
\usepackage{amsmath, amssymb}
\usepackage{amscd}
\usepackage[latin2]{inputenc}
\usepackage{t1enc}
\usepackage[mathscr]{eucal}
\usepackage{indentfirst}
\usepackage{graphicx}
\usepackage{graphics}
\usepackage{pict2e}
\usepackage{epic}
\usepackage{color}
\numberwithin{equation}{section}
\usepackage[margin=2.9cm]{geometry}
\usepackage{epstopdf}

\theoremstyle{plain}
\newtheorem{thm}{Theorem}[section]
\newtheorem{lem}[thm]{Lemma}
\newtheorem{cor}[thm]{Corollary}
\newtheorem{prop}[thm]{Proposition}

\theoremstyle{definition}
\newtheorem{defn}[thm]{Definition}

\newtheorem{rem}[thm]{Remark}
\newtheorem{?}[thm]{Problem}

\theoremstyle{definition}
\newtheorem*{nt*}{Notation}

\theoremstyle{plain}
\newtheorem{asu}{Assumption}[section]

\newcommand{\inner}[2]{\left\langle {#1}, {#2} \right\rangle}
\newcommand{\norm}[1]{\left\lVert {#1} \right\rVert}
\newcommand{\abs}[1]{\left\lvert {#1} \right\rvert}

\newcommand{\calC}{\mathcal C}
\newcommand{\calE}{\mathcal E}

\newcommand {\C} {\mathbb C}
\newcommand {\R} {\mathbb R}
\newcommand {\Z} {\mathbb Z}

\begin{document}

\title[]{The gradient descent method from the perspective of fractional calculus}

\author{Pham Viet Hai}%
\address[P. V. Hai]{Faculty of Mathematics, Mechanics and Informatics, University of Science, Vietnam National University, Hanoi, Vietnam.}%
\address[P. V. Hai]{Thang Long Institute of Mathematics and Applied Sciences, Thang Long University, Hanoi, Vietnam.}
\email{phamviethai@hus.edu.vn}

\author{Joel A. Rosenfeld}
\address[Joel A. Rosenfeld]{Department of Mathematics and Statistics, University of South Florida, Tampa, Florida.}
\email{rosenfeldj@usf.edu}



\begin{abstract}
Motivated by gradient methods in optimization theory, we give methods based on $\psi$-fractional derivatives of order $\alpha$ in order to solve unconstrained optimization problems. The convergence of these methods is analyzed in detail. This paper also presents an Adams-Bashforth-Moulton (ABM) method for the estimation of solutions to equations involving $\psi$-fractional derivatives. Numerical examples using the ABM method show that the fractional order $\alpha$ and weight $\psi$ are tunable parameters, which can be helpful for improving the performance of gradient descent methods.
\end{abstract}

\maketitle

\section{Introduction}
\subsection{Optimization Problems}
Given a continuously differentiable function $f:\R^d\to\R$, we are interested in solving the unconstrained optimization problem:
\begin{equation}\label{op-pro}
\min_{x\in\R^d}f(x),
\end{equation}
where $d\in\Z_{\geq 0}$. From classical multivariate calculus, if the objective function $f(\cdot)$ attains its local minimum at $y_*$, then $y_*$ is a \emph{stationary point}, that is
\begin{equation}\label{op-condition}
\nabla f(y_*)=0.
\end{equation}
Furthermore, if it is also assumed that $f(\cdot)$ is convex, then any stationary point must be a global minimum of $f(\cdot)$. Thus, \eqref{op-pro} may be reduced to the problem of finding stationary points. If the description of the gradient $\nabla f(\cdot)$ is simple, one can seek all stationary points and, among them, all global minima, by solving \eqref{op-condition}. However, in practice solving \eqref{op-condition} may be difficult.

Gradient descent is the modal first-order iterative algorithm for solving \eqref{op-pro}. Low computational complexity makes gradient descent an ideal algorithm for large-scale problems with medium accuracy. Gradient descent originates from the observation that if the objective function, $f(\cdot)$, is defined and differentiable in a neighborhood of a point $\omega$, then the direction of greatest decrease from $\omega$ is the negative gradient of $f$ at $\omega$. With this observation in mind, the gradient descent algorithm starts with a prediction $x_0$ for a local minimum of $f(\cdot)$, and constructs the sequence $x(0),x(1),x(2),\cdots$ such that
\begin{equation}\label{diff-eq-dis}
x(k+1)=x(k)-\beta(k)\nabla f(x(k)),\quad x(0)=x_0\in\R^d,
\end{equation}
where the \emph{step size} $\beta:\Z_{\geq 0}\to\R_{\geq 0}$ is allowed to vary at every iteration $k$. The result of the gradient descent algorithm is a monotonic sequence
$$
f(x(0))\geq f(x(1))\geq f(x(2))\geq\cdots
$$
and it is expected that the sequence $\{x(k)\}_{k=0}^\infty$ converges to a local minimum. In practice, the stopping criterion is usually of the form $\norm{\nabla f(x)}\leq\varepsilon$, where $\epsilon > 0$. More details can be found in the monographs \cite{BV, YN, ruszczynski2006nonlinear}.

Rearranging \eqref{diff-eq-dis} results as
$$
\dfrac{x(k+1)-x(k)}{\beta(k)}=-\nabla f(x(k)).
$$
Hence, \eqref{diff-eq-dis} may be viewed as the discretization of the ordinary differential equation
\begin{equation}\label{GM-ODE}
y'(t)=-\nabla f(y(t)),\quad y(0)=y_0\in\R^d,
\end{equation}
by the explicit Euler scheme with step size $\beta(k)$. The system \eqref{GM-ODE} is called \emph{continuous gradient method}. The analogy between difference equations and differential equations has been recognized and exploited very regularly. Many results related to difference equations can carry over to corresponding results for differential equations and vice versa. For instance, solutions of equations \eqref{diff-eq-dis} and \eqref{GM-ODE} converge to the unique stationary point at an exponential rate under the assumption that $f(\cdot)$ is both strongly convex and smooth \cite{ruszczynski2006nonlinear, scieur2017integration}. Furthermore, if a convergence result is proved for a continuous method, then various finite difference schemes for the solution of this Cauchy problem may be constructed.

Although in many applications of optimization the discrete time method is modally employed, continuous gradient methods are employed frequently in the design of feedback controllers for nonlinear dynamical systems. These methods are combined with Lyapunov techniques to give theoretical guarantees for the stability of closed loop control systems (cf. \cite{dixon2013nonlinear}), and in the case of an unconstraint affine-quadratic regulator, the optimal controller is expressed in terms of the negative gradient of the optimal value function (cf. \cite[Equation 1.13]{kamalapurkar2018reinforcement}).

The advantages of fractional calculus (FC) have been recognized and exploited very regularly by many researchers. Although in the early stage FC was considered as a branch of pure mathematics, in recent decades it spreads into many other areas of science. Applications to viscoelastic models and connections to anomalous relaxation and rheological models can be found in  \cite{gao2017general, yang2017new, zbMATH07051303, yang2019newgeneralcalculi, yang2019newnonconventional, yang2019newmathematicalmodels, yang2017newlocal, zbMATH07092994}. The present manuscript employs methods from fractional order calculus to improve the convergence rate of continuous time optimization methods. Improvements in convergence rates over continuous time gradient methods using classical calculus are realized through Mittag-Leffler convergence rates using $\psi-$fractional derivatives.

\subsection{Fractional calculus}
While inquiry into the FC dates back to the inception of differential calculus, serious study and rigorous definitions did not appear until the ninteenth and twentieth centuries. 
Several monographs and surveys describe the progress in the area of FC. Prominent are the books \cite{diethelm2010analysis, kilbas2006north} as well as the survey \cite{valerio2013fractional}, which includes a historical review and notes concerning scientists that contributed to the development of FC. Some more recent work concerning fractional calculus can be found in \cite{rosenfeld2017approximating,yang2019newnonconventional,yang2019newgeneralcalculi,yang2017new,liu2020generalized, yang2017exactlocal, yang2018newcomputational, yang2015new}.

The following are distinct approaches to FC by Riemann-Liouville and Hadamard.
\begin{defn}\label{def-1}
Given $0<\alpha\leq 1$, the \emph{Riemann-Liouville integral} of a function $x$ is defined as
$$
I_a^\alpha x(t)\triangleq\dfrac{1}{\Gamma(\alpha)}\int\limits_a^t (t-\tau)^{\alpha-1}x(\tau)\,d\tau,
$$
where $\Gamma(\alpha)\triangleq\int\limits_0^\infty\tau^{\alpha-1}e^{-\tau}\,d\tau$. The \emph{Riemann-Liouville and Caputo fractional derivatives} are defined as, respectively
$$
{}^{RL}\!D_a^\alpha\triangleq\frac{d}{dt}\circ I_a^{1-\alpha}\text{, and } {}^C\!D_a^\alpha\triangleq I_a^{1-\alpha}\circ\frac{d}{dt}.
$$
\end{defn}

\begin{defn}\label{def-2}
Given $0<\alpha\leq 1$, the \emph{Hadamard integral} of a function $x$ is defined as
$$
H_1^\alpha x(t)\triangleq\dfrac{1}{\Gamma(\alpha)}\int\limits_1^t \left(\ln\dfrac{t}{\tau}\right)^{\alpha-1}x(\tau)\,\dfrac{d\tau}{\tau}.
$$
The \emph{Hadamard fractional derivative} is defined as
$$
{}^H\!D_1^\alpha\triangleq t\frac{d}{dt}\circ H_1^{1-\alpha}.
$$
\end{defn}

In \cite{chen2017study, wei2020generalization}, a fractional order gradient method was developed. Owing to the ``memory'' characteristic of fractional and nonlocal derivatives, the gradient method designed in this way does not converge to the real extreme point. To overcome the disadvantage on memory of fractional derivatives, three viable solutions were developed, containing the fixed memory step, the higher order truncation, and the variable fractional order. Numerical examples in \cite{chen2017study, wei2020generalization} show that all of the developed methods converge quickly.

In \cite{liang2019fractional}, Liang et al. developed a fractional differential equation, which generalizes equation \eqref{GM-ODE}. Specifically, these authors replaced the usual derivative $y'(\cdot)$ in \eqref{GM-ODE} with Caputo fractional derivatives ${}^C\!D_a^\alpha$, where $0<\alpha<1$. The fractional order is a adjustable parameter, which can be helpful for improving the performance. Both theoretical analysis and numerical experiments in \cite{liang2019fractional} reveal that this fractional differential equation may possess faster or slower convergence rate than equation \eqref{GM-ODE}, depending on specific problems. We emphasize that the paper \cite{liang2019fractional} deals only with Caputo fractional derivatives and the cases of Riemann-Liouville or Hadamard are not considered.

\subsection{Aim and Content}
Recently, several concepts of fractional derivatives have been proposed, studied and applied to practical problems. In \cite{almeida2017caputo}, the definition of the \emph{$\psi$-fractional derivatives} unifies a large class of fractional derivatives. Leveraging the $\psi$-derivative, the present manuscript provides fractional differential equations that resolve a class of unconstrained optimization problems. Significantly, the convergence rate of these differential equations is analyzed in Section \ref{sec-RL} and Section \ref{caputo-type-sec}. Convergence rates outperforming integer order gradient descent methods are obtained in Theorem \ref{1} through the selection of particular $\psi$. In the case of non-strongly convex $f(\cdot)$, a convergence rate of $O(\psi(t)^{-\lambda})$ is obtained. Subsequently, under the assumption on the strong convexity, Mittag-Leffler convergence, a general type of exponential convergence, is demonstrated, and conditions when an exponential convergence occurs are also established.

The rest of the paper is organized as follows. Section \ref{pre} is devoted to recalling basic knowledge of convex analysis and $\psi$-fractional derivatives. What makes these $\psi$-fractional derivatives interesting is the fact that they are really generalizations of the well-known concepts, including Definition \ref{def-1} and \ref{def-2}. Section \ref{sec-auxi} contains several technical observations which will be later on referred to, including the chain rule, and a Jensen-type inequality. Motivated by gradient methods in optimization theory, in Section \ref{sec-RL} we give fractional differential equations of Riemann-Liouville type to solve unconstrained optimization problems. In parallel, we study fractional differential equations of Caputo type in Section \ref{caputo-type-sec}. Section \ref{de-num-me} gives a generalization of the ABM method of \cite{diethelm2010analysis} for $\psi$-Caputo derivatives, and this numerical method is utilized to implement the optimization procedures of this manuscript. The paper concludes with suggestions of future research in Section \ref{conl-sec}.

\subsection{Notations}
Throughout the paper, we denote by $\Z$, $\R$, $\C$ by the set of integers, real numbers, complex numbers, respectively. For a set $A\subseteq\R$, $A_{\geq\delta}$ stands for the set $\{x\in A: x\geq\delta\}$. 
Let $\R^d$ be the set of real $d$-dimensional vectors endowed with the Euclidean inner product $\inner{\cdot}{\cdot}$ and the standard Euclidean norm $\|\cdot\|$. The gradient of the function $f(\cdot)$ at $x$ is denoted as $\nabla f(x)$. For two symmetric matrices $A,B\in\R^{d\times d}$, writing $A\preceq B$ means that $\langle Ax,x\rangle\leq \langle Bx,x\rangle$ for all $x\in\R^d$. 

\section{Preliminaries}\label{pre}
\subsection{Convex analysis}
This section provides an exposition on convex functions that will be leveraged in the sequel.
\begin{defn}
A function $f:\R^d\to\R$ is of class $\calC^m$ for $m\in\Z_{\geq 0}$ (written $f(\cdot)\in \calC^m$) if the partial derivatives
$$
\dfrac{\partial^{m_1+\cdots +m_d} f}{\partial x_1^{m_1}\cdots\partial x_d^{m_d}}
$$
all exist and are continuous for every selection $m_1,\ldots, m_d$ with $m_1+\cdots+m_d\leq m$.
\end{defn}

Several assumptions concerning the objective function will be used, including Lipschitz conditions on the gradient, which may be achieved by bounding the Hessian of a twice differentiable function, convexity of the objective function, and a uniform lower bound on the Hessian. These assumptions are summarized in Assumption \ref{asu-f}. Assumption \ref{Sfne0} is also included to ensure that the optimization problem is well posed.
\begin{asu}\label{asu-f}
The function $f:\R^d\to\R$ is
\begin{enumerate}
\item of class $\calC^1$ and $L_f$-smooth, i.e.
$$
\norm{\nabla f(x)-\nabla f(y)}\leq L_f\norm{x-y},\quad\forall x,y\in\R^d.
$$
\item of class $\calC^1$ and convex, i.e.
$$
f(x)\geq f(y)+\langle \nabla f(y),x-y\rangle,\quad\forall x,y\in\R^d.
$$
\item of class $\calC^2$ and there exists a constant $m_f>0$ such that
\[
m_f I_d\preceq\nabla^2f(x),\quad\forall x\in\R^d.
\]
\end{enumerate}
\end{asu}

\begin{asu}\label{Sfne0}
The set $S(f)\triangleq\{z\in\R^d: \nabla f(z)=0\}\ne\emptyset$.
\end{asu}

\begin{defn}
A function $f \in \calC^1$ is \emph{strongly convex} with parameter $m_f>0$ if
\begin{equation}\label{str-conv-2}
f(x)\geq f(y)+\langle \nabla f(y),x-y\rangle+\frac{m_f}{2}\|x-y\|^2,\quad\forall x,y\in\R^d.
\end{equation}
\end{defn}

Convexity has several useful consequences.
\begin{lem}[{\cite[Ineq. (2.1.7)-(2.1.9)]{YN}}]
If the function $f(\cdot)$ is convex and $M_f$-smooth, then
\begin{equation}\label{str-conv-1}
f(x)+\inner{\nabla f(x)}{y-x}+\dfrac{1}{2M_f}\norm{\nabla f(x)-\nabla f(y)}^2\leq f(y),\quad\forall x,y\in\R^d,
\end{equation}
\begin{equation}\label{Nesterov-ineq-218}
\langle \nabla f(x)-\nabla f(y),x-y\rangle\geq\dfrac{1}{M_f}\|\nabla f(x)-\nabla f(y)\|^2,\quad\forall x,y\in\R^d,
\end{equation}
and
\begin{equation}\label{Nesterov-ineq-219}
\inner{\nabla f(x)-\nabla f(y)}{x-y}\leq L\norm{x-y}^2,\quad\forall x,y\in\R^d.
\end{equation}
\end{lem}

\subsection{Fractional calculus}
In this section, some basic definitions and techniques related to $\psi$-fractional calculus are presented. To simplify notation, we denote
$$
J_\psi\triangleq \dfrac{1}{\psi'(t)}\cdot\dfrac{d}{dt}.
$$

\begin{defn}[\cite{almeida2017caputo}]\label{defn-frac-deri}
Let $0<\alpha\leq 1$. The \emph{$\psi$-Riemann-Liouville integral} of a function $x$ is defined by
$$
(I_{a,\psi}^\alpha x)(t)\triangleq\dfrac{1}{\Gamma(\alpha)}\int\limits_a^t (\psi(t)-\psi(\tau))^{\alpha-1}\psi'(\tau)x(\tau)\,d\tau.
$$
The \emph{$\psi$-Riemann-Liouville and $\psi$-Caputo fractional derivatives} are defined as, respectively
$$
{}^{RL}\!D_{a,\psi}^\alpha\triangleq J_\psi\circ I_{a,\psi}^{1-\alpha}\text{ and }^C\!D_{a,\psi}^\alpha\triangleq I_{a,\psi}^{1-\alpha}\circ J_\psi.
$$
\end{defn}

\begin{defn}[\cite{almeida2017caputo}]
The \emph{$\psi$-Riemann-Liouville and $\psi$-Caputo fractional derivatives} of a $d$-dimensional vector function $\mathbf{x}(t)=(x_1(t),\cdots,x_d(t))^T$ are defined component-wise as
$$
{}^{RL}\!D_{a,\psi}^\alpha\mathbf{x}(t)\triangleq({}^{RL}\!D_{a,\psi}^\alpha x_1(t),\cdots,{}^{RL}\!D_{a,\psi}^\alpha x_d(t))^T,
$$
$$
{}^C\!D_{a,\psi}^\alpha\mathbf{x}(t)\triangleq({}^C\!D_{a,\psi}^\alpha x_1(t),\cdots,{}^C\!D_{a,\psi}^\alpha x_d(t))^T,
$$
respectively.
\end{defn}

\begin{rem}
Several well known fractional derivatives manifest as particular cases of $\psi$-fractional derivatives. For example, given appropriate choices of the kernel, $\psi$, we obtain the Caputo fractional derivative (when $\psi(t)=t$) and Hadamard fractional derivative (when $\psi(t)=\ln t$).
\end{rem}

In \cite{almeida2018fractional}, Almeida et al. investigated fractional differential equations with the $\psi$-Caputo derivative:
\begin{equation}\label{FO-DE}
{}^C\!D_{a,\psi}^\alpha x(t)=g(t,x(t)),\quad t\geq a,\quad x(a)=x_0,
\end{equation}
where $\alpha\in (0,1]$, and $g:\R_{\geq 0}\times\R^d\to\R^d$ is a continuous vector-valued function. The results regarding the existence and uniqueness of solutions of \eqref{FO-DE} were established by using fixed point theorems, which agrees with the approach for establishing existence and uniqueness in the classical setting (cf. \cite{coddington1955theory}).

\begin{defn}
\normalfont
A continuous function $x:\R_{\geq 0}\to\R^d$ is called a \emph{solution} of equation \eqref{FO-DE} if it satisfies this equation for every $t\geq a$. In this case, $x_0$ is called the \emph{initial value} of the solution $x(\cdot)$.
\end{defn}

A closely related fractional differential equation to \eqref{FO-DE} arises from the Riemann-Liouville setting,
\begin{equation}\label{FO-DE-RL}
{}^{RL}\!D_{a,\psi}^\alpha y(t)=g(t,y(t)),\quad t\geq a,\quad I_{a,\psi}^{1-\alpha}y(a+)=y_0,
\end{equation}
where a function is called a solution to \eqref{FO-DE-RL} if it satisfies the equation for all $t > a$ as was the case for \eqref{FO-DE}.

\begin{prop}\label{inte-eq-equiv}
Let $0<\alpha\leq 1$. Then the following assertions hold.
\begin{enumerate}
\item A continuous function $x(\cdot)$ is a solution of equation \eqref{FO-DE} if and only if it satisfies
$$
x(t)=x(a)+\dfrac{1}{\Gamma(\alpha)}\int\limits_a^t(\psi(t)-\psi(s))^{\alpha-1}\psi'(s)g(s,x(s))\,ds,\quad t\geq a.
$$
\item A continuous function $y(\cdot)$ is a solution of equation \eqref{FO-DE-RL} if and only if it satisfies
$$
y(t)=\dfrac{y_0}{\Gamma(\alpha)}(\psi(t)-\psi(a))^{\alpha-1}+\dfrac{1}{\Gamma(\alpha)}\int\limits_a^t(\psi(t)-\psi(s))^{\alpha-1}\psi'(s)g(s,y(s))\,ds,\quad t\geq a.
$$
\end{enumerate}
\end{prop}
\begin{proof}
The first item was proved in \cite[Theorem 2]{almeida2018fractional}. Meanwhile, the second one can be obtained by using
$$
I_{\ell,\psi}^\alpha\circ{}^{RL}\!D_{\ell,\psi}^\alpha y(t)=y(t)-\dfrac{1}{\Gamma(\alpha)}I_{\ell,\psi}^{1-\alpha}y(a+)[\psi(t)-\psi(\ell)]^{\alpha-1}.
$$
\end{proof}

Similar to the exponential function frequently used in the solutions of integer-order systems, a function frequently used in fractional calculus is the \emph{Mittag-Leffler function} defined as
\[
\calE_{\alpha,\beta}(z)\triangleq\sum\limits_{j\geq 0}\dfrac{z^j}{\Gamma(j\alpha+\beta)},\quad z\in\C,
\]
where $\alpha,\beta>0$. In order to simplify notation, we write $\calE_\alpha(\cdot)$ in stead of $\calE_{\alpha,1}$. The following propositions gather some well-known inequalities on fractional calculus, which are required in later proofs.

\begin{prop}[\cite{schneider1996completely}]\label{E-CM}
The function $\calE_{\alpha,\beta}(-t)$ is completely monotone, that is
\[
(-1)^m\dfrac{d^m}{dt^m}\calE_{\alpha,\beta}(-t)\geq 0,\quad\forall t\geq 0,\forall m\in\Z_{\geq 0},
\]
if and only if $0<\alpha\leq 1$, $\beta\geq\alpha$.
\end{prop}

\begin{prop}[{\cite[Theorem 4.5]{diethelm2010analysis}}]\label{->0}
Let $\alpha>0$ and $\gamma\in\C$. The limit $\lim\limits_{t\to\infty}\calE_{\alpha}(-\gamma t^\alpha)=0$ holds if $\abs{\arg(\gamma)}<\alpha\pi/2$.
\end{prop}

Several assumptions concerning the weight $\psi$ will be used.
\begin{asu}\label{asu-psi}
The function $\psi:\R_{\geq\ell}\to\R_{\geq 0}$ is 
\begin{enumerate}
\item strictly increasing, of class $\calC^1$ with $\psi'>0$.
\item $\sup\{\psi(t):t\geq\ell\}=\infty$.
\item of class $\calC^2$ with $\psi''\geq 0$.
\end{enumerate}
\end{asu}

\section{Auxiliary results}\label{sec-auxi}
This section contains several technical observations which will be later on referred to.

\subsection{Chain rule}
In this subsection, we show that the fractional-order derivative of a composite function is different from its integer-order derivative.

\begin{prop}\label{lem-C}
Let $\alpha\in (0,1]$ and $\ell\in\R_{\geq 0}$. For $t\geq\ell$, let us define the function $\zeta_t$ by setting
\begin{equation}\label{zetats}
\zeta_t(s)\triangleq g(f(s))-g(f(t))-\inner{\nabla g(f(t))}{f(s)-f(t)}.
\end{equation}
Then the following identity holds for every $t\geq\ell$
\begin{eqnarray*}
&&\Gamma(1-\alpha)\left({}^C\!D_{\ell,\psi}^\alpha g(f(t))-\inner{\nabla g(f(t))}{{}^C\!D_{\ell,\psi}^\alpha f(t)}\right)\\
&&=-\dfrac{\psi'(\ell)\zeta_t(\ell)}{(\psi(t)-\psi(\ell))^\alpha}-\int\limits_\ell^t\zeta_t(s)\left[\dfrac{\psi''(s)}{(\psi(t)-\psi(s))^\alpha}+\dfrac{\alpha\psi'(s)^2}{(\psi(t)-\psi(s))^{\alpha+1}}\right]\,ds.
\end{eqnarray*}
Consequently, if the function $g(\cdot)$ is convex and if the function $\psi(\cdot)$ satisfies Assumptions \ref{asu-psi}(1,3), then we obtain the following chain rule.
$$
{}^C\!D_{\ell,\psi}^\alpha g(f(t))\leq\inner{\nabla g(f(t))}{{}^C\!D_{\ell,\psi}^\alpha f(t)}.
$$
If the function $g(\cdot)$ is concave (i.e. $-g(\cdot)$ is convex) and if the function $\psi(\cdot)$ satisfies Assumptions \ref{asu-psi}(1,3), then
$$
{}^C\!D_{\ell,\psi}^\alpha g(f(t))\geq\inner{\nabla g(f(t))}{{}^C\!D_{\ell,\psi}^\alpha f(t)}.
$$
\end{prop}
\begin{proof}
Taking into account the definition of ${}^C\!D_{\ell,\psi}^\alpha$, we can write
\begin{eqnarray*}
&&\Gamma(1-\alpha)\left({}^C\!D_{\ell,\psi}^\alpha g(f(t))-\inner{\nabla g(f(t))}{{}^C\!D_{\ell,\psi}^\alpha f(t)}\right)\\
&&=\int\limits_\ell^t(\psi(t)-\psi(s))^{-\alpha}\psi'(s)\inner{\nabla g(f(s))-\nabla g(f(t))}{f'(s)}\,ds\\
&&=\int\limits_\ell^t(\psi(t)-\psi(s))^{-\alpha}\psi'(s)\zeta_t'(s)\,ds\\
&&=\int\limits_\ell^t(\psi(t)-\psi(s))^{-\alpha}\psi'(s)\,d\zeta_t(s),
\end{eqnarray*}
which implies, by integration parts, that
\begin{eqnarray*}
&&\Gamma(1-\alpha)\left({}^C\!D_{\ell,\psi}^\alpha g(f(t))-\inner{\nabla g(f(t))}{{}^C\!D_{\ell,\psi}^\alpha f(t)}\right)\\
&&=\lim\limits_{s\to t}(\psi(t)-\psi(s))^{-\alpha}\psi'(s)\zeta_t(s)-(\psi(t)-\psi(\ell))^{-\alpha}\psi'(\ell)\zeta_t(\ell)\\
&&-\int\limits_\ell^t\zeta_t(s)[(\psi(t)-\psi(s))^{-\alpha}\psi''(s)+\alpha\psi'(s)^2 (\psi(t)-\psi(s))^{-\alpha-1}]\,ds.
\end{eqnarray*}
Thus, we obtain the desired result because
\begin{eqnarray*}
\lim\limits_{s\to t}(\psi(t)-\psi(s))^{-\alpha}\zeta_t(s)
&=&\lim\limits_{s\to t}\dfrac{\zeta_t(s)}{(\psi(t)-\psi(s))^\alpha}\\
&=&\lim\limits_{s\to t}\dfrac{\zeta_t'(s)}{-\alpha\psi'(s)(\psi(t)-\psi(s))^{\alpha-1}}\quad\text{(by L'Hospital's Rule)}\\
&=&0\quad\text{(as $0<\alpha\leq 1$)}.
\end{eqnarray*}
\end{proof}

\begin{rem}
Proposition \ref{lem-C} generalizes existing inequalities in \cite{aguila2014lyapunov, chen2017convex} and contains both existing works as special cases. Specially, if $\psi(t)=t$, Proposition \ref{lem-C} is reduced to the original inequality in \cite[Theorem 1]{chen2017convex}.
\end{rem}

\begin{prop}\label{prop-H}
Let $\alpha\in (0,1]$ and $b\in\R_{\geq 0}$. For $t\geq b$, let $\zeta_t$ be the function given in \eqref{zetats}. Then the following identity holds for every $t\geq b$
\begin{eqnarray*}
&&\Gamma(1-\alpha)\left({}^{RL}\!D_{b,\psi}^\alpha g(f(t))-\inner{\nabla g(f(t))}{{}^{RL}\!D_{b,\psi}^\alpha f(t)}\right)\\
&&=[\psi(t)-\psi(b)]^{-\alpha}[g(f(t))-\inner{\nabla g(f(t))}{f(t)}]\\
&&-\alpha\int\limits_b^t[\psi(t)-\psi(s)]^{-\alpha-1}\psi'(s)\zeta_t(s)\,ds.\\
\end{eqnarray*}
\end{prop}
\begin{proof}
By the Newton-Leibniz formula, one has
$$
f(t)=f(b)+\int\limits_b^t f'(s)\,ds=f(b)+I_{b,\psi}^1(f'/\psi')(t),
$$
and so by \cite[Corollary 1]{jarad2019generalized}
\begin{eqnarray*}
{}^{RL}\!D_{b,\psi}^\alpha f(t)
&=&{}^{RL}\!D_{b,\psi}^\alpha[f(b)](t)+{}^{RL}\!D_{b,\psi}^\alpha\circ I_{b,\psi}^1(f'/\psi')(t)\\
&=&\dfrac{f(b)}{\Gamma(1-\alpha)}[\psi(t)-\psi(b)]^{-\alpha}+I_{b,\psi}^{1-\alpha}(f'/\psi')(t)\\
&=&\dfrac{1}{\Gamma(1-\alpha)}\left[f(b)[\psi(t)-\psi(b)]^{-\alpha}+\int\limits_b^t[\psi(t)-\psi(s)]^{-\alpha}f'(s)\,ds\right].
\end{eqnarray*}
Hence, we can write
\begin{eqnarray*}
&&\Gamma(1-\alpha)\left({}^{RL}\!D_{\ell,\psi}^\alpha g(f(t))-\inner{\nabla g(f(t))}{{}^{RL}\!D_{\ell,\psi}^\alpha f(t)}\right)\\
&&=[\psi(t)-\psi(b)]^{-\alpha}[g(f(b))-\inner{\nabla g(f(t))}{f(b)}]\\
&&+\int\limits_b^t[\psi(t)-\psi(s)]^{-\alpha}\inner{\nabla g(f(s))-\nabla g(f(t))}{f'(s)}\,ds\\
&&=[\psi(t)-\psi(b)]^{-\alpha}[g(f(b))-\inner{\nabla g(f(t))}{f(b)}]\\
&&+\int\limits_b^t[\psi(t)-\psi(s)]^{-\alpha}\,d\zeta_t(s),
\end{eqnarray*}
which implies, by integration parts, that
\begin{eqnarray*}
&&\Gamma(1-\alpha)\left({}^{RL}\!D_{\ell,\psi}^\alpha g(f(t))-\inner{\nabla g(f(t))}{{}^{RL}\!D_{\ell,\psi}^\alpha f(t)}\right)\\
&&=[\psi(t)-\psi(b)]^{-\alpha}[g(f(b))-\inner{\nabla g(f(t))}{f(b)}]+\lim\limits_{s\to t}[\psi(t)-\psi(s)]^{-\alpha}\zeta_t(s)\\
&&-[\psi(t)-\psi(b)]^{-\alpha}\zeta_t(b)-\alpha\int\limits_b^t[\psi(t)-\psi(s)]^{-\alpha-1}\psi'(s)\zeta_t(s)\,ds\\
&&=[\psi(t)-\psi(b)]^{-\alpha}[g(f(t))-\inner{\nabla g(f(t))}{f(t)}]\\
&&-\alpha\int\limits_b^t[\psi(t)-\psi(s)]^{-\alpha-1}\psi'(s)\zeta_t(s)\,ds.
\end{eqnarray*}
\end{proof}

Based on Proposition \ref{prop-H}, we can directly obtain the following corollary.
\begin{cor}\label{lem-H}
Let $\alpha\in (0,1]$, $b\in\R_{\geq 0}$ and $\psi$ be a function satisfying Assumption \ref{asu-psi}(1). If the function $g$ satisfies Assumption \ref{asu-f}(2), then
$$
{}^{RL}\!D_{b,\psi}^\alpha g(f(t))\leq\inner{\nabla g(f(t))}{{}^{RL}\!D_{b,\psi}^\alpha f(t)}+\dfrac{g(0)}{\Gamma(1-\alpha)}(\psi(t)-\psi(b))^{-\alpha}.
$$
\end{cor}
\begin{proof}
It follows from Assumption \ref{asu-f}(2), that
$$
g(f(t))-\inner{\nabla g(f(t))}{f(t)}\leq g(0)-\inner{\nabla g(f(t))}{0}=g(0).
$$
Hence, we can use Proposition \ref{prop-H} to get the desired result.
\end{proof}

\begin{rem}
Proposition \ref{prop-H} generalises existing inequalities in \cite{liu2016asymptotical} and contains both existing works as its special cases. Specifically, if $g(x)=\frac{\norm{x}^2}{2}$ and $\psi(t)=t$, Proposition \ref{prop-H} is reduced to the original inequality in \cite[Theorem 1]{liu2016asymptotical}.
\end{rem}

\subsection{Jensen-type inequality}
Denote $[c,\ell]_d\triangleq [c,\ell]\times\cdots\times [c,\ell]$ ($d$ times). The following result can be viewed as the Jensen-type inequality for fractional integral.
\begin{prop}\label{jense-inq}
Let $\varphi: [c,\ell]_d\to\R$ be a convex function and $h:[a,b]\to [c,\ell]_d$. Then
$$
\varphi\left(\dfrac{\Gamma(\alpha+1)}{(\psi(t)-\psi(a))^\alpha}\cdot (I_{a,\psi}^\alpha h)(t)\right)\leq \dfrac{\Gamma(\alpha+1)}{(\psi(t)-\psi(a))^\alpha}\cdot(I_{a,\psi}^\alpha\varphi\circ h)(t).
$$
\end{prop}
\begin{proof}
Let $h=(h_1,\cdots,h_d)$. By the assumption,
$$
c\leq h_j(s)\leq\ell,\quad\forall s\in [a,b],\forall j\in\{1,\cdots d\}.
$$
Multiplying the above inequality by $(\psi(t)-\psi(s))^{\alpha-1}\psi'(s)/\Gamma(\alpha)$ and integrating the resulting inequality with respect to $s$ over $[a,t]$, we obtain
$$
\dfrac{c}{\Gamma(\alpha)}\int\limits_a^t (\psi(t)-\psi(s))^{\alpha-1}\psi'(s)\,ds\leq (I_{a,\psi}^\alpha h_j)(t)\leq\dfrac{\ell}{\Gamma(\alpha)}\int\limits_a^t (\psi(t)-\psi(s))^{\alpha-1}\psi'(s)\,ds.
$$
Subsequently,
$$
c\leq \dfrac{\Gamma(\alpha+1)}{(\psi(t)-\psi(a))^\alpha}\cdot (I_{a,\psi}^\alpha h_j)(t)\leq\ell,\quad\forall a\leq t\leq b,\forall j\in\{1,\cdots d\}.
$$
From the convexity of $\varphi$, we obtain
\begin{eqnarray*}
\varphi(h(t))
&\geq&\varphi\left(\dfrac{\Gamma(\alpha+1)}{(\psi(\tau)-\psi(a))^\alpha}\cdot (I_{a,\psi}^\alpha h)(\tau)\right)\\
&&+\left(h(t)-\dfrac{\Gamma(\alpha+1)}{(\psi(\tau)-\psi(a))^\alpha}\cdot (I_{a,\psi}^\alpha h)(\tau)\right)\varphi'\left(\dfrac{\Gamma(\alpha+1)}{(\psi(\tau)-\psi(a))^\alpha}\cdot (I_{a,\psi}^\alpha h)(\tau)\right).
\end{eqnarray*}
Multiplying the above inequality by $(\psi(\tau)-\psi(t))^{\alpha-1}\psi'(t)/\Gamma(\alpha)$ and integrating the resulting inequality with respect to $t$ over $[a,\tau]$, we obtain
\begin{eqnarray*}
(I_{a,\psi}^\alpha\varphi\circ h)(\tau)
&\geq&\dfrac{(\psi(\tau)-\psi(a))^\alpha}{\Gamma(\alpha+1)}\cdot \varphi\left(\dfrac{\Gamma(\alpha+1)}{(\psi(\tau)-\psi(a))^\alpha}\cdot (I_{a,\psi}^\alpha h)(\tau)\right)\\
&&+\left((I_{a,\psi}^\alpha h)(\tau)-(I_{a,\psi}^\alpha h)(\tau)\right)\varphi'\left(\dfrac{\Gamma(\alpha+1)}{(\psi(\tau)-\psi(a))^\alpha}\cdot (I_{a,\psi}^\alpha h)(\tau)\right)\\
&=&\dfrac{(\psi(\tau)-\psi(a))^\alpha}{\Gamma(\alpha+1)}\cdot\varphi\left(\dfrac{\Gamma(\alpha+1)}{(\psi(\tau)-\psi(a))^\alpha}\cdot (I_{a,\psi}^\alpha h)(\tau)\right).
\end{eqnarray*}
\end{proof}

\subsection{Mittag-Leffler function}
The following result will be leveraged in the proofs of Theorems \ref{exp-rate-1} and \ref{exp-rate-2}.
\begin{lem}\label{limpsiEalpha=Gamma/L}
Let $\alpha\in (0,1]$ and $\ell\geq 0$. Suppose that the function $\psi:\R_{\geq\ell}\to\R_{\geq 0}$ satisfies Assumptions \ref{asu-psi}(1-2). Then the limit
$$
\lim\limits_{t\to\infty}\int\limits_\ell^t(\psi(t)-\psi(s))^{\alpha-1}\psi'(s)\calE_\alpha(-L(\psi(s)-\psi(\ell))^\alpha)\,ds=\dfrac{\Gamma(\alpha)}{L}
$$
holds.
\end{lem}
\begin{proof}
It was indicated in \cite[Lemma 2]{almeida2017caputo}, that the function $\calE_\alpha(-L(\psi(\cdot)-\psi(\ell))^\alpha)$ is the solution of the initial value problem
\begin{equation}\label{eq:linear-equation}
{}^C\!D_{\ell,\psi}^\alpha u(t)=-Lu(t),\quad t\geq\ell,\quad u(\ell)=1.
\end{equation}
Hence by Proposition \ref{inte-eq-equiv}, we have
$$
\calE_\alpha(-L(\psi(t)-\psi(\ell))^\alpha)=1-\dfrac{L}{\Gamma(\alpha)}\int\limits_\ell^t(\psi(t)-\psi(s))^{\alpha-1}\psi'(s)\calE_\alpha(-L(\psi(s)-\psi(\ell))^\alpha)\,ds,
$$
and hence we make use of Proposition \ref{->0} in order to get the desired result.
\end{proof}

\section{Gradient methods with Riemann-Liouville type fractional derivatives}\label{sec-RL}
In this subsection, we study
\begin{equation}\label{frac-diff-eq-con-RL}
{}^{RL}\!D_{\ell,\psi}^\alpha z(t)=-\beta\nabla f(z(t)),\quad (I_{\ell,\psi}^{1-\alpha}z)(\ell^+)=z_0\in\R^d,\quad \forall t\geq\ell.
\end{equation}

For $y_*\in S(f)$, let us define the function $\varphi:\R_{\geq\ell}\to\R_{\geq 0}$ by setting
\begin{equation}\label{func-varphi}
\varphi(t)\triangleq \dfrac{1}{2}\norm{z(t)-y_*}^2,\quad t\geq\ell.
\end{equation}

\subsection{The case of convex $f(\cdot)$}
This subsection is devoted to investigating \eqref{frac-diff-eq-con-RL} under the assumption on the non-strongly convexity of $f(\cdot)$. 
\begin{thm}
Let $\alpha\in (0,1]$. Suppose that the function $f:\R^d\to\R$ satisfies Assumption \ref{asu-f}(2) and $\psi:\R_{\geq\ell}\to\R_{\geq 0}$ satisfies Assumption \ref{asu-psi}(1). Consider the fractional-order differential equation, \eqref{frac-diff-eq-con-RL}, where the step size $\beta$ is a constant. Then we have the following conclusions.
\begin{enumerate}
\item There exists a constant $C$ such that
\begin{equation}\label{convex-non-strong-res}
\norm{z(t)-y_*}^2\leq C[\psi(t)-\psi(\ell)]^{\alpha-1},\quad\forall t\geq\ell.
\end{equation}
\item If $\alpha=1$, then $z(\cdot)$ is bounded.
\end{enumerate}
\end{thm}
\begin{proof}
By Proposition \ref{lem-H} and Assumption \ref{asu-f}(2), we have
\begin{eqnarray*}
{}^{RL}\!D_{\ell,\psi}^\alpha\varphi(t)
\leq-\beta\inner{z(t)-y_*}{\nabla f(z(t))}\leq\beta [f_*-f(z(t))]\leq 0,\quad\forall t\in\R_{\geq\ell}
\end{eqnarray*}
and so by Assumption \ref{asu-psi}(1),
\begin{eqnarray*}
0\geq\beta I_{\ell,\psi}^\alpha (f_*-f\circ z)(t)\geq I_{\ell,\psi}^\alpha\circ{}^{RL}\!D_{\ell,\psi}^\alpha\varphi(t)=\varphi(t)-\dfrac{1}{\Gamma(\alpha)}I_{\ell,\psi}^{1-\alpha}\varphi(\ell+)[\psi(t)-\psi(\ell)]^{\alpha-1},
\end{eqnarray*}
where in the last equality we use \cite[Theorem 2.6]{jarad2019generalized}.
\end{proof}

\subsection{The case of strongly convex $f(\cdot)$}
It turns out under the assumption of strong convexity, solutions of  \eqref{frac-diff-eq-con-RL} admit Mittag-Leffler convergence, which is a general type of exponential convergence to a stationary point.
\begin{thm}\label{1-}
Let $\alpha\in (0,1]$. Suppose that the function $f:\R^d\to\R$ satisfies Assumption \ref{asu-f}(3) and $\psi:\R_{\geq\ell}\to\R_{\geq 0}$ satisfies Assumption \ref{asu-psi}(1). Consider the fractional-order differential equation, \eqref{frac-diff-eq-con-RL}, where the step size $\beta$ is a constant. Then we have the following conclusions.
\begin{enumerate}
\item The solution $z(\cdot)$ converges to $y_*$ with rate:
\[\norm{z(t)-y_*}^2\leq\varphi_0[\psi(t)-\psi(\ell)]^{\alpha-1}\calE_{\alpha,\alpha}(-\beta m_f(\psi(t)-\psi(\ell))^\alpha),\quad\forall t\geq\ell,\]
where $\varphi_0\triangleq \frac{1}{2}(I_{\ell,\psi}^{1-\alpha}\norm{z-y_*}^2)(\ell^+)$.
\item If we additionally assume that Assumption \ref{asu-f}(1) hold, then
\[
f(z(t))-f(y_*)\leq\dfrac{1}{2}M_f \varphi_0[\psi(t)-\psi(\ell)]^{\alpha-1}\calE_{\alpha,\alpha}(-\beta m_f(\psi(t)-\psi(\ell))^\alpha),\quad\forall t\geq\ell.
\]
\end{enumerate}
\end{thm}
\begin{proof}
(1) Note that
\[
f(x)\geq f(y)+\inner{\nabla f(y)}{x-y}+\dfrac{m_f}{2}\norm{x-y}^2,\quad\forall x,y\in\R^d.
\]
In particular with $y=z(t)$ and $x=y_*$, we get
\begin{eqnarray*}
\inner{z(t)-y_*}{\nabla f(z(t))}
&\geq& f(z(t))-f(y_*)+\dfrac{m_f}{2}\norm{z(t)-y_*}^2\\
&\geq& \dfrac{m_f}{2}\norm{z(t)-y_*}^2,\quad\forall t\geq\ell,
\end{eqnarray*}
where the last inequality holds since $y_*\in\underset{x\in\R^d}{\text{arg\,min}}f(x)$. By Proposition \ref{lem-H}, we have
\begin{eqnarray*}
{}^{RL}\!D_{\ell,\psi}^\alpha\varphi(t)
&\leq& \inner{z(t)-y_*}{{}^{RL}\!D_{\ell,\psi}^\alpha z(t)}\\
&=&-\beta\inner{z(t)-y_*}{\nabla f(z(t))}\\
&\leq&-\dfrac{\beta m_f}{2}\norm{z(t)-y_*}^2.
\end{eqnarray*}
For setting
$$
h(t)\triangleq -\dfrac{\beta m_f}{2}\norm{z(t)-y_*}^2-{}^{RL}\!D_{\ell,\psi}^\alpha\varphi(t),
$$
we have $h(t)\geq 0$ for all $t\geq 0$, and moreover
\begin{equation}\label{eq-1}
{}^{RL}\!D_{\ell,\psi}^\alpha\varphi(t)=-\beta m_f\varphi(t)-h(t).
\end{equation}
By \cite[Theorem 5.1]{jarad2019generalized}, we can write
\begin{eqnarray*}
\varphi(t)
&=&\varphi_0[\psi(t)-\psi(\ell)]^{\alpha-1}\calE_{\alpha,\alpha}(-\beta m_f(\psi(t)-\psi(\ell))^\alpha)\\
&&-\int\limits_\ell^t (\psi(t)-\psi(\tau))^{\alpha-1}\psi'(\tau)\calE_{\alpha,\alpha}(-\beta m_f(\psi(t)-\psi(\tau))^\alpha)h(\tau)\,d\tau\\
&\leq&\varphi_0[\psi(t)-\psi(\ell)]^{\alpha-1}\calE_{\alpha,\alpha}(-\beta m_f(\psi(t)-\psi(\ell))^\alpha),
\end{eqnarray*}
where in the last inequality, we use Proposition \ref{E-CM}.

(2) By inequalities \eqref{str-conv-1} and \eqref{Nesterov-ineq-219}, we get
\[
f(z(t))-f(y_*)\leq\dfrac{M_f}{2}\norm{z(t)-y_*}^2\leq\dfrac{1}{2}M_f \varphi_0[\psi(t)-\psi(\ell)]^{\alpha-1}\calE_{\alpha,\alpha}(-\beta m_f(\psi(t)-\psi(\ell))^\alpha).
\]
\end{proof}

\subsection{Convergence at an exponential rate}
Theorem \ref{1-} reveals that the solution of  \eqref{frac-diff-eq-con-RL} can converge to a stationary point at the Mittag-Leffler convergence rate. In particular with $\alpha=1$, we recover the exponential rate $O(e^{-\beta m_f\psi(t)})$ for the continuous gradient method \eqref{GM-ODE}. In this subsection, we study the exponential rate when $\alpha\in (0,1)$.
\begin{thm}\label{exp-rate-1}
Let $\alpha\in (0,1)$. Suppose that the function $f:\R^d\to\R$ satisfies Assumption \ref{asu-f}(1) and $\psi:\R_{\geq\ell}\to\R_{\geq 0}$ satisfies Assumptions \ref{asu-psi}(1-2). Consider the fractional-order differential equation, \eqref{frac-diff-eq-con-RL}, where the step size $\beta$ is constant. If the solution $z(\cdot)$ of  \eqref{frac-diff-eq-con-C} converges to $y_*\in S(f)$ at the exponential rate $O(e^{-\omega\psi(t)})$, then $y_*=0$.
\end{thm}
\begin{proof}
Let $z(\cdot)$ be a solution of  \eqref{frac-diff-eq-con-C} converging to a stationary point $y_*$ with the convergence rate $O(e^{-\omega\psi(t)})$. Then there exists $t_1\geq\ell$ such that
$$
\norm{z(t)-y_*}\leq e^{-\omega\psi(t)},\quad\forall t\geq t_1.
$$
Assume by way of contradiction that $y_*\ne 0$ and so we can set
$$
K\triangleq\dfrac{\beta}{\norm{y_*}}+1.
$$
Since $\calE_\alpha(x)=2\calE_{2\alpha}(x^2)-\calE_\alpha(-x)$, by \cite[Corollary 3.8]{gorenflo2014mittag}, we can find $t_2\geq t_1$ with the property that
$$
\calE_\alpha(-L_f(\psi(t)-\psi(\ell))^\alpha)>Ke^{-\omega\psi(t)},\quad\forall t\geq t_2.
$$
Denote
$$
Q\triangleq\sup\{\norm{z(t)-y_*}:t\in [\ell,t_2]\},\quad u(t)\triangleq z_0(\psi(t)-\psi(\ell))^{\alpha-1}-\Gamma(\alpha)y_*.$$
By Proposition \ref{inte-eq-equiv}, $z(\cdot)$ is of the following form
$$
z(t)=\dfrac{z_0}{\Gamma(\alpha)}(\psi(t)-\psi(\ell))^{\alpha-1}-\dfrac{\beta}{\Gamma(\alpha)}\int\limits_\ell^t (\psi(t)-\psi(s))^{\alpha-1}\psi'(s)\nabla f(z(s))\,ds,
$$
and so
\begin{eqnarray*}
\Gamma(\alpha)\norm{u(t)}
&\leq&\Gamma(\alpha)\norm{z(t)-y_*}+\beta\int\limits_\ell^t (\psi(t)-\psi(s))^{\alpha-1}\psi'(s)\norm{\nabla f(z(s))}\,ds\\
&\leq&\Gamma(\alpha)\norm{z(t)-y_*}+\beta L_f\int\limits_\ell^t (\psi(t)-\psi(s))^{\alpha-1}\psi'(s)\norm{z(s)-y_*}\,ds\\
&=&\Gamma(\alpha)\norm{z(t)-y_*}+\beta L_f\left\{\int\limits_\ell^{t_2}+\int\limits_{t_2}^t\right\} (\psi(t)-\psi(s))^{\alpha-1}\psi'(s)\norm{z(s)-y_*}\,ds.
\end{eqnarray*}
We estimate
\begin{eqnarray*}
\Gamma(\alpha)\norm{u(t)}
&\leq&\Gamma(\alpha)\norm{z(t)-y_*}+\beta L_f Q\int\limits_\ell^{t_2}(\psi(t)-\psi(s))^{\alpha-1}\psi'(s)\,ds\\
&&+\dfrac{\beta L_f}{K}\int\limits_{t_2}^t(\psi(t)-\psi(s))^{\alpha-1}\psi'(s)\calE_\alpha(-L_f(\psi(s)-\psi(\ell))^\alpha)\,ds\\
&\leq&\Gamma(\alpha)\norm{z(t)-y_*}+\dfrac{\beta L_f Q}{\alpha}[(\psi(t)-\psi(\ell))^\alpha-(\psi(t)-\psi(t_2))^\alpha]\\
&&+\dfrac{\beta L_f}{K}\int\limits_\ell^t(\psi(t)-\psi(s))^{\alpha-1}\psi'(s)\calE_\alpha(-L_f(\psi(s)-\psi(\ell))^\alpha)\,ds.
\end{eqnarray*}
Note that by the Mean Value theorem (applied for the function $(\psi(t)-\psi(\cdot))^\alpha$), there is $\delta\in (\ell,t_2)$ with
$$
(\psi(t)-\psi(\ell))^\alpha-(\psi(t)-\psi(t_2))^\alpha=-\alpha(\psi(t)-\psi(\delta))^{\alpha-1}\psi'(\delta),
$$
which implies, as $\alpha\in (0,1)$ and $\psi$ satisfies Assumption \ref{asu-psi}(2), that
$$
\lim\limits_{t\to\infty}[(\psi(t)-\psi(\ell))^\alpha-(\psi(t)-\psi(t_2))^\alpha]=0.
$$
This limit and Lemma \ref{limpsiEalpha=Gamma/L} show
\begin{eqnarray*}
\lim\limits_{t\to\infty}\norm{u(t)}
&\leq&\lim\limits_{t\to\infty}\left\{\Gamma(\alpha)\norm{z(t)-y_*}+\dfrac{\beta L_f Q}{\alpha}[(\psi(t)-\psi(\ell))^\alpha-(\psi(t)-\psi(t_2))^\alpha]\right\}\\
&&+\dfrac{\beta L_f}{K}\lim\limits_{t\to\infty}\int\limits_\ell^t(\psi(t)-\psi(s))^{\alpha-1}\psi'(s)\calE_\alpha(-L_f(\psi(s)-\psi(\ell))^\alpha)\,ds\\
&=&\dfrac{\beta\Gamma(\alpha)}{K}.
\end{eqnarray*}
Thus,
$$
\dfrac{\beta\Gamma(\alpha)}{K}\geq\lim\limits_{t\to\infty}\norm{u(t)}=\Gamma(\alpha)\norm{y_*},
$$
but this is a contradiction.
\end{proof}

\subsection{Perturbations}
In this section, we study
\begin{equation}\label{frac-diff-eq-con-RL-per}
{}^{RL}\!D_{\ell,\psi}^\alpha z(t)=-\beta\nabla f(z(t))+g(t),\quad (I_{\ell,\psi}^{1-\alpha}z)(\ell^+)=z_0\in\R^d,\quad \forall t\geq\ell,
\end{equation}
where the function $g(\cdot)$ reflects an external action on the system. Thus, \eqref{frac-diff-eq-con-RL-per} can be viewed as a perturbation of \eqref{frac-diff-eq-con-RL}.

\begin{asu}\label{asu-g}
The function $g:\R_{\geq\ell}\to\R^d$ satisfies
$$
Q\triangleq\sup_{t\geq\ell}\int\limits_\ell^t(\psi(t)-\psi(s))^{\alpha-1}\psi'(s)\norm{g(s)}^2\,dt<\infty.
$$
\end{asu}

\begin{thm}
Let $\alpha\in (0,1]$. Suppose that the function $f:\R^d\to\R$ satisfies Assumption \ref{asu-f}(3) (in this case, $S(f)=\{y_*\}$) and $\psi:\R_{\geq\ell}\to\R_{\geq 0}$ satisfies Assumptions \ref{asu-psi}(1-2). Consider the fractional-order differential equation, \eqref{frac-diff-eq-con-RL-per}, where the step size $\beta$ is a constant. If the step size $\beta$ satisfies
\begin{equation}\label{beta>}
\beta>\dfrac{1}{m_f},
\end{equation}
then $y_*$ is a limit point of $z(\cdot)$, i.e. there exists a sequence $\{s_m\}\in\R_{\geq\ell}$ such that
$$
\lim\limits_{m\to\infty}z(s_m)=y_*.
$$
\end{thm}
\begin{proof}
By Proposition \ref{lem-C}, we have
\begin{eqnarray*}
{}^{RL}\!D_{\ell,\psi}^\alpha\varphi(t)
&\leq& \inner{z(t)-y_*}{{}^{RL}\!D_{\ell,\psi}^\alpha z(t)}=-\beta\inner{z(t)-y_*}{\nabla f(z(t))}+\inner{z(t)-y_*}{g(t)}\\
&\leq&-\dfrac{\beta m_f}{2}\norm{z(t)-y_*}^2+\inner{z(t)-y_*}{g(t)}\\
&\leq&-\dfrac{1}{2}(\beta m_f-1)\norm{z(t)-y_*}^2+\dfrac{1}{2}\norm{g(t)}^2.
\end{eqnarray*}
Set
$$
h(t)\triangleq-\dfrac{1}{2}(\beta m_f-1)\norm{z(t)-y_*}^2+\dfrac{1}{2}\norm{g(t)}^2-{}^{RL}\!D_{\ell,\psi}^\alpha\varphi(t).
$$
Thus, $h(t)\geq 0$ for $t\geq\ell$, and
$$
{}^{RL}\!D_{\ell,\psi}^\alpha\varphi(t)=-(\beta m_f-1)\varphi(t)+\dfrac{1}{2}\norm{g(t)}^2-h(t).
$$
Hence,
\begin{eqnarray*}
\varphi(t)-\dfrac{1}{\Gamma(\alpha)}[\psi(t)-\psi(\ell)]^{\alpha-1}I_{\ell,\psi}^{1-\alpha}\varphi(\ell+)
&=&I_{\ell,\psi}^\alpha\circ{}^{RL}\!D_{\ell,\psi}^\alpha\varphi(t)\\
&=&-(\beta m_f-1)I_{\ell,\psi}^\alpha\varphi(t)+\dfrac{1}{2}I_{\ell,\psi}^\alpha\norm{g}^2(t)-I_{\ell,\psi}^\alpha h(t)\\
&\leq&-(\beta m_f-1)I_{\ell,\psi}^\alpha\varphi(t)+\dfrac{1}{2}I_{\ell,\psi}^\alpha\norm{g}^2(t),
\end{eqnarray*}
which gives
\begin{eqnarray*}
(\beta m_f-1)I_{\ell,\psi}^\alpha\varphi(t)
&\leq&\dfrac{1}{2}I_{\ell,\psi}^\alpha\norm{g}^2(t)+\dfrac{1}{\Gamma(\alpha)}[\psi(t)-\psi(\ell)]^{\alpha-1}I_{\ell,\psi}^{1-\alpha}\varphi(\ell+)-\varphi(t)\\
&\leq&\dfrac{1}{2}I_{\ell,\psi}^\alpha\norm{g}^2(t)+\dfrac{1}{\Gamma(\alpha)}[\psi(t)-\psi(\ell)]^{\alpha-1}I_{\ell,\psi}^{1-\alpha}\varphi(\ell+).
\end{eqnarray*}
By Assumption \ref{asu-g}, we can write
\begin{eqnarray*}
&&\int\limits_\ell^t(\psi(t)-\psi(s))^{\alpha-1}\psi'(s)\norm{z(s)-y_*}^2\,ds\\
&&\leq\dfrac{Q+\dfrac{2}{\Gamma(\alpha)}[\psi(t)-\psi(\ell)]^{\alpha-1}I_{\ell,\psi}^{1-\alpha}\varphi(\ell+)}{\beta m_f-1},
\end{eqnarray*}
which implies, by Assumption \ref{asu-psi}(2), that
$$
\inf\{\norm{z(s)-y_*}:s\geq\ell\}=0.
$$
\end{proof}

\section{Gradient methods with Caputo type}\label{caputo-type-sec}
In this section, we study
\begin{equation}\label{frac-diff-eq-con-C}
{}^C\!D_{\ell,\psi}^\alpha z(t)=-\beta\nabla f(z(t)),\,z(\ell)=z_0\in\R^d,\quad \forall t\geq\ell.
\end{equation}

For $y_*\in S(f)$, let us define the function $\lambda:\R_{\geq\ell}\to\R_{\geq 0}$ by setting
\begin{equation}\label{func-psi}
\lambda(t)\triangleq \dfrac{1}{2}\norm{z(t)-y_*}^2,\quad t\geq\ell.
\end{equation}

\subsection{The case of convex $f(\cdot)$}
This subsection is devoted to studying \eqref{frac-diff-eq-con-C} under the assumption on that $f(\cdot)$ is convex but perhaps not strongly convex. Theorem \ref{fnxm} indicates that there exists a limit point of $z(\cdot)$ which belongs to $S(f)$. 

\begin{thm}\label{fnxm}
Let $\alpha\in (0,1]$. Suppose that the function $f:\R^d\to\R$ satisfies Assumption \ref{asu-f}(2) and $\psi:\R_{\geq\ell}\to\R_{\geq 0}$ satisfies Assumptions \ref{asu-psi}(1-2). Consider the fractional-order differential equation, \eqref{frac-diff-eq-con-RL}, where the step size $\beta$ is a constant. Then $z(\cdot)$ is bounded, and there is a sequence $\{z(s_m)\}$ converging to a point in $S(f)$. Furthermore, $f(\widehat{z}(t))$ converges to $f_*$ as $O(\psi(t)^{-\alpha})$, where
$$
\widehat{z}(t)\triangleq\dfrac{\Gamma(\alpha+1)}{(\psi(t)-\psi(\ell))^\alpha}\cdot (I_{\ell,\psi}^\alpha z)(t).
$$
\end{thm}
\begin{proof}
By Proposition \ref{lem-C}, we have
\begin{eqnarray*}
{}^C\!D_{\ell,\psi}^\alpha\lambda(t)
\leq-\beta\inner{z(t)-y_*}{\nabla f(z(t))}\leq\beta [f_*-f(z(t))]\leq 0,\quad\forall t\in\R_{\geq\ell}
\end{eqnarray*}
and so
\begin{eqnarray*}
0\geq\beta I_{\ell,\psi}^\alpha (f_*-f\circ z)(t)\geq I_{\ell,\psi}^\alpha\circ{}^C\!D_{\ell,\psi}^\alpha\lambda(t)=\lambda(t)-\lambda(\ell),
\end{eqnarray*}
where in the last equality we use \cite[Theorem 2.6 and equality 16]{jarad2019generalized}. As a result, $z(\cdot)$ is bounded and furthermore
$$
\sup\{I_{\ell,\psi}^\alpha (f\circ z-f_*)(t):t\geq\ell\}<\infty.
$$
Taking into account the form of $I_{\ell,\psi}^\alpha$ in Definition \ref{defn-frac-deri} and Assumption \ref{asu-psi}(2), we must have
$$
f_*=\inf\{f(z(s)):s\geq\ell\},
$$
which implies, by the definition of an infimum, that there exists a sequence $\{s_n\}\subseteq\R_{\geq\ell}$ such that $\lim\limits_{n\to\infty}f(z(s_n))=f_*$. Since $z(\cdot)$ is bounded, we can extract a convergent subsequence $\{z(s_{n_k})\}$ of $\{z(s_n)\}$. Suppose that $\omega\triangleq\lim\limits_{k\to\infty} z(s_{n_k})$ and then
$$
f_*=\lim\limits_{n\to\infty}f(z(s_n))=\lim\limits_{k\to\infty}f(z(s_{n_k}))=f(\omega);
$$
namely, $\omega$ is a stationary point.

Since $f(\cdot)$ is convex, it follows from Proposition \ref{jense-inq} that
$$
\dfrac{\Gamma(\alpha+1)}{(\psi(t)-\psi(\ell))^\alpha}\cdot (I_{\ell,\psi}^\alpha f\circ z)(t)\geq f(\widehat{z}(t)).
$$
Thus,
\begin{eqnarray*}
0\leq f(\widehat{z}(t))-f_*\leq\dfrac{\Gamma(\alpha+1)}{(\psi(t)-\psi(\ell))^\alpha}\cdot I_{\ell,\psi}^\alpha (f\circ z-f_*)(t)\leq\dfrac{\Gamma(\alpha+1)\lambda(\ell)}{\beta(\psi(t)-\psi(\ell))^\alpha},
\end{eqnarray*}
which implies at least an $O(\psi(t)^{-\alpha})$ convergence rate.
\end{proof}

\subsection{The case of strongly convex $f(\cdot)$}
Under the assumption on the strong convexity, the following result establishes Mittag-Leffler convergence to the optimal point, which is a general type of exponential convergence.

\begin{thm}\label{1}
Let $\alpha\in (0,1]$. Suppose that the function $f:\R^d\to\R$ satisfies Assumption \ref{asu-f}(3) and $\psi:\R_{\geq\ell}\to\R_{\geq 0}$ satisfies Assumption \ref{asu-psi}(1). Consider the fractional-order differential equation, \eqref{frac-diff-eq-con-C}, where the step size $\beta$ is a constant. Then we have the following conclusions.
\begin{enumerate}
\item The solution $z(\cdot)$ converges to $y_*$ with rate:
\[\norm{z(t)-y_*}^2\leq\norm{z_0-y_*}^2\calE_\alpha(-\beta m_f(\psi(t)-\psi(\ell))^\alpha),\quad\forall t\geq\ell.\]
\item If we additionally assume that Assumption \ref{asu-f}(1) holds, then
\[
f(z(t))-f(y_*)\leq\dfrac{1}{2}M_f\norm{z_0-y_*}^2\calE_\alpha(-\beta m_f(\psi(t)-\psi(\ell))^\alpha),\quad\forall t\geq\ell.
\]
\end{enumerate}
\end{thm}
\begin{proof}
(1) 
By Proposition \ref{lem-C}, we have
\begin{eqnarray*}
{}^C\!D_{\ell,\psi}^\alpha\lambda(t)
&\leq& \inner{z(t)-y_*}{{}^C\!D_{\ell,\psi}^\alpha z(t)}\\
&=&-\beta\inner{z(t)-y_*}{\nabla f(z(t))}\\
&\leq&-\dfrac{\beta m_f}{2}\norm{z(t)-y_*}^2,\quad\text{(by Assumption \ref{asu-f}(3))}.
\end{eqnarray*}
For setting
$$
h(t)\triangleq -\dfrac{\beta m_f}{2}\norm{z(t)-y_*}^2-{}^C\!D_{\ell,\psi}^\alpha\lambda(t),
$$
we have $h(t)\geq 0$ for all $t\geq\ell$, and moreover
\begin{equation}\label{eq-1-C}
{}^C\!D_{\ell,\psi}^\alpha\lambda(t)=-\beta m_f\lambda(t)-h(t).
\end{equation}
By \cite[Theorem 5.2]{jarad2019generalized}, we can write
\begin{eqnarray*}
\lambda(t)
&=&\lambda(\ell)\calE_{\alpha}(-\beta m_f(\psi(t)-\psi(\ell))^\alpha)\\
&&-\int\limits_\ell^t (\psi(t)-\psi(\tau))^{\alpha-1}\psi'(\tau)\calE_{\alpha,\alpha}(-\beta m_f(\psi(t)-\psi(\tau))^\alpha)h(\tau)\,d\tau\\
&\leq&\lambda(\ell)\calE_{\alpha}(-\beta m_f(\psi(t)-\psi(\ell))^\alpha),
\end{eqnarray*}
where in the last inequality, we use Proposition \ref{E-CM}.

(2) The proof is similar to the second item of Theorem \ref{1}.
\end{proof}

\subsection{Convergence at an exponential rate}
Theorem \ref{1} reveals that the solution of \eqref{frac-diff-eq-con-C} can converge to a stationary point at a Mittag-Leffler convergence rate. In particular with $\alpha=1$, this convergence speed reduces to the exponential rate $O(e^{-\beta m_f\psi(t)})$. The following result indicates that if $\alpha\in (0,1)$, then there is no nontrivial solution of  \eqref{frac-diff-eq-con-C} converging to a stationary point with such exponential rate.

\begin{thm}\label{exp-rate-2}
Let $\alpha\in (0,1)$. Suppose that the function $f:\R^d\to\R$ satisfies Assumption \ref{asu-f}(1) and $\psi:\R_{\geq\ell}\to\R_{\geq 0}$ satisfies Assumptions \ref{asu-psi}(1-2). Consider the fractional-order differential equation, \eqref{frac-diff-eq-con-C}, where the step size $\beta$ is constant. Then every solution of equation \eqref{frac-diff-eq-con-C} does not converge to any stationary point at the exponential rate $O(e^{-\omega\psi(t)})$.
\end{thm}
\begin{proof}
Assume by way of contradiction that $z(\cdot)$ is a solution of equation \eqref{frac-diff-eq-con-C} which converges to a stationary point $y_*$ with the convergence rate $O(e^{-\omega\psi(t)})$. Then there exists $t_1\geq\ell$ such that
$$
\norm{z(t)-y_*}\leq e^{-\omega\psi(t)},\quad\forall t\geq t_1.
$$
Denote
$$
K\triangleq\dfrac{\beta}{\norm{z(\ell)-y_*}}+1.
$$
Since $\calE_\alpha(x)=2\calE_{2\alpha}(x^2)-\calE_\alpha(-x)$, by \cite[Corollary 3.8]{gorenflo2014mittag}, we can find $t_2\geq t_1$ with the property that
$$
\calE_\alpha(-L_f(\psi(t)-\psi(\ell))^\alpha)>Ke^{-\omega\psi(t)},\quad\forall t\geq t_2.
$$
Denote
$$Q\triangleq\sup\{\norm{z(t)-y_*}:t\in [\ell,t_2]\}.$$
By Proposition \ref{inte-eq-equiv}, $z(\cdot)$ is of the following form
$$
z(t)=z_0-\dfrac{\beta}{\Gamma(\alpha)}\int\limits_\ell^t (\psi(t)-\psi(s))^{\alpha-1}\psi'(s)\nabla f(z(s))\,ds,
$$
and so
\begin{eqnarray*}
\Gamma(\alpha)\norm{z_0-y_*}
&\leq&\Gamma(\alpha)\norm{z(t)-y_*}+\beta L_f\left\{\int\limits_\ell^{t_2}+\int\limits_{t_2}^t\right\} (\psi(t)-\psi(s))^{\alpha-1}\psi'(s)\norm{z(s)-y_*}\,ds.
\end{eqnarray*}
We estimate
\begin{eqnarray*}
\Gamma(\alpha)\norm{z_0-y_*}
&\leq&\Gamma(\alpha)\norm{z(t)-y_*}+\beta L_f Q\int\limits_\ell^{t_2}(\psi(t)-\psi(s))^{\alpha-1}\psi'(s)\,ds\\
&&+\dfrac{\beta L_f}{K}\int\limits_{t_2}^t(\psi(t)-\psi(s))^{\alpha-1}\psi'(s)\calE_\alpha(-L_f(\psi(s)-\psi(\ell))^\alpha)\,ds\\
&\leq&\Gamma(\alpha)\norm{z(t)-y_*}+\dfrac{\beta L_f Q}{\alpha}[(\psi(t)-\psi(\ell))^\alpha-(\psi(t)-\psi(t_2))^\alpha]\\
&&+\dfrac{\beta L_f}{K}\int\limits_\ell^t(\psi(t)-\psi(s))^{\alpha-1}\psi'(s)\calE_\alpha(-L_f(\psi(s)-\psi(\ell))^\alpha)\,ds.
\end{eqnarray*}
Since
$$
\lim\limits_{t\to\infty}[(\psi(t)-\psi(\ell))^\alpha-(\psi(t)-\psi(t_2))^\alpha]=0,
$$
by Lemma \ref{limpsiEalpha=Gamma/L} we have
\begin{eqnarray*}
\Gamma(\alpha)\norm{z_0-y_*}
&\leq&\lim\limits_{t\to\infty}\left\{\Gamma(\alpha)\norm{z(t)-y_*}+\dfrac{\beta L_f Q}{\alpha}[(\psi(t)-\psi(\ell))^\alpha-(\psi(t)-\psi(t_2))^\alpha]\right\}\\
&&+\dfrac{\beta L_f}{K}\lim\limits_{t\to\infty}\int\limits_\ell^t(\psi(t)-\psi(s))^{\alpha-1}\psi'(s)\calE_\alpha(-L_f(\psi(s)-\psi(\ell))^\alpha)\,ds\\
&=&\dfrac{\beta\Gamma(\alpha)}{K}.
\end{eqnarray*}
This contradicts the assumption.
\end{proof}

\subsection{Perturbations}
In this section, we study
\begin{equation}\label{frac-diff-eq-con-C-pe}
{}^C\!D_{\ell,\psi}^\alpha z(t)=-\beta\nabla f(z(t))+g(t),\,z(\ell)=z_0\in\R^d,\quad \forall t\geq\ell,
\end{equation}
where the function $g(\cdot)$ reflects an external action on the system. Equation \eqref{frac-diff-eq-con-RL-per} can be viewed as a perturbation of \eqref{frac-diff-eq-con-RL}.
\begin{thm}
Let $\alpha\in (0,1]$. Suppose that the function $f:\R^d\to\R$ satisfies Assumption \ref{asu-f}(3) (in this case, $S(f)=\{y_*\}$) and $\psi:\R_{\geq\ell}\to\R_{\geq 0}$ satisfies Assumption \ref{asu-psi}(1-2). Consider the fractional-order differential equation \eqref{frac-diff-eq-con-C-pe}, where the step size $\beta$ is a constant. If the step size $\beta$ satisfies \eqref{beta>}, then $y_*$ is a limit point of $z(\cdot)$, i.e. there exists a sequence $\{s_m\}\in\R_{\geq\ell}$ such that
$$
\lim\limits_{m\to\infty}z(s_m)=y_*.
$$
\end{thm}
\begin{proof}
By Proposition \ref{lem-C}, we have
\begin{eqnarray*}
{}^C\!D_{\ell,\psi}^\alpha\lambda(t)
&\leq& \inner{z(t)-y_*}{{}^C\!D_{\ell,\psi}^\alpha z(t)}=-\beta\inner{z(t)-y_*}{\nabla f(z(t))}+\inner{z(t)-y_*}{g(t)}\\
&\leq&-\dfrac{\beta m_f}{2}\norm{z(t)-y_*}^2+\inner{z(t)-y_*}{g(t)}\\
&\leq&-\dfrac{1}{2}(\beta m_f-1)\norm{z(t)-y_*}^2+\dfrac{1}{2}\norm{g(t)}^2.
\end{eqnarray*}
For setting
$$
h(t)\triangleq-\dfrac{1}{2}(\beta m_f-1)\norm{z(t)-y_*}^2+\dfrac{1}{2}\norm{g(t)}^2-{}^C\!D_{\ell,\psi}^\alpha\lambda(t),
$$
we have $h(t)\geq 0$ for $t\geq\ell$ and
$$
{}^C\!D_{\ell,\psi}^\alpha\lambda(t)=-(\beta m_f-1)\lambda(t)+\dfrac{1}{2}\norm{g(t)}^2-h(t).
$$
Hence,
\begin{eqnarray*}
\lambda(t)-\lambda(\ell)
&=&I_{\ell,\psi}^\alpha\circ{}^C\!D_{\ell,\psi}^\alpha\lambda(t)\\
&=&-(\beta m_f-1)I_{\ell,\psi}^\alpha\lambda(t)+\dfrac{1}{2}I_{\ell,\psi}^\alpha\norm{g}^2(t)-I_{\ell,\psi}^\alpha h(t)\\
&\leq&-(\beta m_f-1)I_{\ell,\psi}^\alpha\lambda(t)+\dfrac{1}{2}I_{\ell,\psi}^\alpha\norm{g}^2(t),
\end{eqnarray*}
which gives
\begin{eqnarray*}
(\beta m_f-1)I_{\ell,\psi}^\alpha\lambda(t)
&\leq&\dfrac{1}{2}I_{\ell,\psi}^\alpha\norm{g}^2(t)+\lambda(\ell)-\lambda(t)\\
&\leq&\dfrac{1}{2}I_{\ell,\psi}^\alpha\norm{g}^2(t)+\lambda(\ell).
\end{eqnarray*}
By Assumption \ref{asu-g}, we can write
$$
\int\limits_\ell^t(\psi(t)-\psi(s))^{\alpha-1}\psi'(s)\norm{z(s)-y_*}^2\,ds\leq\dfrac{Q+2\lambda(\ell)}{\beta m_f-1},\quad\forall t\geq\ell,
$$
which implies, by Assumption \ref{asu-psi}, that
$$
\inf\{\norm{z(s)-y_*}:s\geq\ell\}=0.
$$
\end{proof}

\section{The Design of the Numerical Methods}\label{de-num-me}
In the previous sections, we showed $\psi$-fractional derivative based methods are also a tool solving Problem \eqref{op-pro}. A problem arising is how to simulate these methods in practice. This section studies that problem in the more general context of fractional differential equations by extending the Adams-Bashforth-Moulton (ABM) method to $\psi$-fractional derivatives. After proving the convergence of the ABM method, we offer numerical examples and indicate that \emph{the fractional order $\alpha$ and weight $\psi$ can be adjusted in order to improve the performance}.

ABM methods are a standard approach to numerical methods for resolving IVPs in fractional order systems, and these methods have been employed successfully in works such as \cite{diethelm2010analysis}. To further validate the use of the ABM method, this section also considers a Picard method that is suitable for application to a subset of the problems considered in this manuscript. In particular, note that for $\psi(t) = t^k$ and $r > 0$, $I_{0,\psi}^\alpha t^r = \frac{\Gamma\left(\frac{r}{k}+1\right)}{\Gamma\left(\frac{r}{k} + \alpha + 1\right)} t^{r+\alpha k}$. Hence, if $\alpha k \in \mathbb{Z}_{\geq 0}$, $^CD_{0,\psi}^\alpha$ maps polynomial to polynomials by linearity. When $g$ is a polynomial, this fact is conducive to iteratively applying Picard's method as $\varphi_{m+1}(t) = x(0) + \frac{1}{\Gamma(\alpha)}\int_0^t(\psi(t)-\psi(s))^{\alpha - 1} \psi'(s)g(s,\varphi_m(s)) ds,$ with $\phi_0(t) = x_0$. The Picard iteration is guaranteed to converge in a neighborhood of the origin by way of the proof of the existence and uniqueness theorem.

\subsection{ABM method for $\psi$-fractional derivatives}
Given a function $g:[0,T] \times \mathbb{R}^d \to \mathbb{R}^d$ that is Lipschitz continuous with Lipschitz constant $L$, a solution of equation \eqref{FO-DE} is given as
$$
x(t) = x(a) + \frac{1}{\Gamma(\alpha)} \int_{a}^t (\psi(t) - \psi(s))^{\alpha - 1} \psi'(s) g(s,x(s)) ds.
$$

The Adams-Bashforth-Moulton method for this initial value problem is a predictor corrector method motivated by the work of Diethelm in \cite{diethelm2010analysis}. Just as in \cite{diethelm2010analysis}, the numerical method uses a piecewise constant predictor and a piecewise linear corrector. The advantage of utilizing this approach to the design of a numerical method for the fractional differential equations at hand is that the only substantial adjustments to the ABM method of \cite{diethelm2010analysis} is an alteration of the coefficients employed in the numerical method. Moreover, in contrast to work such as \cite{rosenfeld2017approximating} is that the evaluation of the ABM method at each time-step is $O(k)$ in computation time, where $k$ is the $k$-th step. Whereas, in \cite{rosenfeld2017approximating}, each timestep requires matrix inversion for interpolation, which typically requires $O(k^3)$ computation time. While the ABM method has slower convergence, the advantage gained through computation time makes it more practical for problems with longer time horizons.

Let $0 = t_0 < t_1 < \ldots < t_{k+1}$. As in \cite{diethelm2010analysis}, let 
\begin{equation}\label{eq:linear-basis}\phi_{j,k+1}(z) \triangleq
\left\{\begin{array}{ll}
\frac{z-t_{j-1}}{t_{j} - t_{j-1}} & \text{ if } t_{j-1} < z \le t_j\\
\frac{t_{j+1}-z}{t_{j+1} - t_{j}} & \text{ if } t_{j} < z \le t_{j+1}\\
0 & \text{ otherwise. }\\
\end{array}\right.
\end{equation}
Given a twice continuously differentiable function $f:[0,t_{k+1}] \to \mathbb{R}$, a piecewise linear expression of $f$ is given as $\tilde f \triangleq \sum_{i=0}^{k+1} f(t_i)\phi_{i,k+1}$, and \[\sup_{t} | f(t) - \tilde f(t) | \le  \max_{i=1,\ldots,k+1} |t_{i} - t_{i-1}|^2 \cdot \sup_{t} |f''(t)|.\] For regularly spaced $t_i$, with step size $h > 0$, it follows that 
\begin{equation}\label{O(h2)}
\sup_{t} | f(t) - \tilde f(t) | \le  h^2 \sup_{t} |f''(t)|.
\end{equation}
If a piecewise constant approximation of $f$ is given as $\hat f \triangleq \sum_{i=0}^{k} f(t_i) \chi_{[t_i,t_{i+1}]}$, where $\chi_{A}$ is the indicator function for the set $A$, then it can be also be shown that
\begin{equation}\label{O(h)}
    \sup_{t} | f(t) - \hat f(t) | \le h \sup_{t} |f'(t)|.
\end{equation}
Given a collection of points $\{ (t_i,x_i) \}_{i=0}^k$, the predictor for $x_{k+1}$ is then given as \begin{equation}\label{eq:predictor} x^P_{k+1} \triangleq x_0 + \frac{1}{\Gamma(\alpha)} \sum_{i=0}^k b_{i,k+1} g(t_i,x_i),\end{equation} where
\begin{align*}b_{i,k+1} \triangleq \int_{t_i}^{t_{i+1}} (\psi(t_{k+1})-\psi(\tau))^{\alpha - 1} \psi'(\tau) d\tau.
\end{align*}
In this setting $t \mapsto g(t,x(t))$ is being approximated by $\hat g = \sum_{i=0}^{k} g(t_i,x_i) \chi_{[t_i,t_{i+1}]}$ which utilizes the approximated points $x_i \approx x(t_i)$. Consequently,
\begin{align*}\| x(t_{k+1}) - x^P_{k+1} \| &= \frac{1}{\Gamma(\alpha)}\bigg\| \int_0^{t_{k+1}} (\psi(t_{k+1})-\psi(\tau))^{\alpha - 1} \psi'(\tau)[g(\tau, x(\tau))\\
&\hspace{0.75in}- \sum_{i=0}^{k} g(t_i,x_i) \chi_{[t_i,t_{i+1}]}(\tau)]d\tau \bigg\|\\
&= \frac{1}{\Gamma(\alpha)} \sum_{i=0}^k \bigg\| \int_{t_{i}}^{t_{i+1}} (\psi(t_{k+1})-\psi(\tau))^{\alpha - 1} \psi'(\tau) [g(\tau, x(\tau))\\
&\hspace{0.75in}- g(t_i,x_i) \chi_{[t_i,t_{i+1}]}(\tau)] d\tau \bigg\|.
\end{align*}
Examining each summand and noting that
$$g(\tau, x(\tau))-g(t_i,x_i) \chi_{[t_i,t_{i+1}]}(\tau)=g(\tau, x(\tau)) - g(t_i,x(t_i)) + g(t_i,x(t_i)) - g(t_i,x_i) \chi_{[t_i,t_{i+1}]}(\tau),$$
we have
\begin{gather*}\left\| \int_{t_{i}}^{t_{i+1}} (\psi(t_{k+1})-\psi(\tau))^{\alpha - 1} \psi'(\tau) \left(g(\tau, x(\tau)) - g(t_i,x_i) \chi_{[t_i,t_{i+1}]}(\tau) \right) d\tau \right\|\\
\le \left\|\int_{t_{i}}^{t_{i+1}} (\psi(t_{k+1})-\psi(\tau))^{\alpha - 1} \psi'(\tau)( g(\tau, x(\tau)) - g(t_i,x(t_i)) ) d\tau\right\| + b_{i,k+1}L\|x(t_i) - x_i\|.
\end{gather*}
Assuming $t \mapsto g(t,x(t))$ is continuously differentiable and employing the mean value theorem, the following is obtained:
\begin{gather*}
\left\| \sum_{i=0}^k \int_{t_{i}}^{t_{i+1}} (\psi(t_{k+1})-\psi(\tau))^{\alpha - 1} \psi'(\tau)( g(\tau, x(\tau)) - g(t_i,x(t_i)) ) d\tau \right\|\\
\le \sup_{t \in [t_0,t_{k+1}]} \left\| \frac{d}{dt} g(t,x(t)) \right\|\sum_{i=0}^k 
\left( \int_{t_{i}}^{t_{i+1}} (\psi(t_{k+1})-\psi(\tau))^{\alpha - 1} \psi'(\tau) (\tau - ih) d\tau\right)\\
\le \sup_{t \in [t_0,t_{k+1}]} \left\| \frac{d}{dt} g(t,x(t)) \right\|
h  \frac{(\psi(t_{k+1}) - \psi(t_0))^\alpha}{\alpha}\\
\end{gather*}
Combining the above results yields
\begin{align*}\| x(t_{k+1}) - x^P_{k+1} \| 
&\le \frac{1}{\Gamma(\alpha)}\left( \sup_{t \in [t_0,t_{k+1}]} \left\| \frac{d}{dt} g(t,x(t)) \right\|
h  \frac{(\psi(t_{k+1}) - \psi(t_0))^\alpha}{\alpha} \right.\\
&\hspace{0.75in}+ \left. L \sup_{i=1,\ldots,k}\|x(t_i) - x_i\| \sum_{i=0}^k b_{i,k+1}\right)\\
&= \frac{1}{\Gamma(\alpha)} \left(\sup_{t \in [t_0,t_{k+1}]} \left\| \frac{d}{dt} g(t,x(t)) \right\|
h  \frac{(\psi(t_{k+1}) - \psi(t_0))^\alpha}{\alpha}\right.\\
&\hspace{0.75in}\left.+ L \sup_{i=1,\ldots,k}\|x(t_i) - x_i\|  \frac{(\psi(t_{k+1}) - \psi(t_0))^\alpha}{\alpha} \right).
\end{align*}
This inequality completes the analysis for the predictor step. The corrector is given as
\begin{equation}\label{eq:corrector}x_{k+1} \triangleq x_0 + \frac{1}{\Gamma(\alpha)} \left( \sum_{i=0}^k a_{i,k+1} g(t_i,x_i) + a_{k+1,k+1} g(t_{k+1},x_{k+1}^P)\right),\end{equation} where \[a_{i,k+1} \triangleq \int_{0}^{t_{k+1}} (\psi(t_{k+1}) - \psi(\tau))^{\alpha-1}\psi'(\alpha)\phi_{i,k+1}(\tau) d\tau.\]
Thus,
\[ a_{i,k+1} =
\left\{
\begin{array}{ll}
\frac{(\psi(t_{k+1}) - \psi(t_0))^{\alpha}}{\alpha} - \frac{1}{t_1 - t_0}\int_{t_0}^{t_1} \frac{(\psi(t_{k+1}) -  \psi(\tau))^{\alpha}}{\alpha} d\tau; & i = 0\\
\frac{(\psi(t_{k+1}) - \psi(t_i))^{\alpha}}{\alpha}  - \frac{(\psi(t_{k+1}) - \psi(t_{i+1}))^\alpha}{\alpha} &\\
\hspace{0.25in}+\frac{1}{t_{i}-t_{i-1}} \int_{t_{i-1}}^{t_i} \frac{(\psi(t_{k+1}) - \psi(\tau))^\alpha}{\alpha} d\tau -  \frac{1}{t_{i+1}-t_{i}} \int_{t_{i}}^{t_{i+1}} \frac{(\psi(t_{k+1}) - \psi(\tau))^\alpha}{\alpha} d\tau& 1 \le i \le k\\
\frac{1}{t_{k+1}-t_k} \int_{t_k}^{t_{k+1}} \frac{(\psi(t_{k+1}) - \psi(\tau))^\alpha}{\alpha} d\tau; & i = k+1.
\end{array}
\right.
\]
\begin{rem}
For practical implementation a simple left or right hand quadrature rule may be employed to give approximate values for $a_{i,k+1}$:
\[ \tilde a_{i,k+1} \triangleq
\left\{
\begin{array}{ll}
\displaystyle\frac{(\psi(t_{k+1}) - \psi(t_0))^{\alpha}}{\alpha} - \frac{(\psi(t_{k+1}) -  \psi(t_0))^{\alpha}}{\alpha}; & i = 0\\
\displaystyle\frac{(\psi(t_{k+1}) - \psi(t_i))^{\alpha}}{\alpha}  - \frac{(\psi(t_{k+1}) - \psi(t_{i+1}))^\alpha}{\alpha} & 1 \le i \le k\\
\displaystyle \frac{(\psi(t_{k+1}) - \psi(t_k))^\alpha}{\alpha}; & i = k+1.
\end{array}
\right.,
\]
\end{rem}
To establish the relevant inequality for the corrector step consider,
\begin{align*}
    \| x(t_{k+1}) - x_{k+1} \| & = \frac{1}{\Gamma(\alpha)} \left\|  \int_{t_0}^{t_{k+1}} (\psi(t_{k+1}) - \psi(\tau))^{\alpha-1} \psi'(\tau) g(\tau,x(\tau)) d\tau\right.\\
    &\hspace{0.5in}\left.- \sum_{i=0}^{k} a_{i,k+1} g(t_i,x_i) - a_{k+1,k+1} g(t_{k+1},x_{k+1}^P) \right\|\\
    &= \frac{1}{\Gamma(\alpha)} \left\|  \int_{t_0}^{t_{k+1}} (\psi(t_{k+1}) - \psi(\tau))^{\alpha-1} \psi'(\tau) g(\tau,x(\tau)) d\tau - \sum_{i=0}^{k+1} a_{i,k+1} g(t_i,x(t_i))\right.\\
    &\hspace{0.5in}\left.+ \sum_{i=0}^{k+1} a_{i,k+1} g(t_i,x(t_i))  - \sum_{i=0}^{k} a_{i,k+1} g(t_i,x_i) - a_{k+1,k+1} g(t_{k+1},x_{k+1}^P) \right\|.
\end{align*}
Leveraging the Lipschitz property of $g$, it follows that 
\begin{gather*}
    \left\|\sum_{i=0}^{k+1} a_{i,k+1} g(t_i,x(t_i))  - \sum_{i=0}^{k} a_{i,k+1} g(t_i,x_i) - a_{k+1,k+1} g(t_{k+1},x_{k+1}^P)\right\|\\
    \le L  \sup_{i=1,\ldots,k} \|x(t_i) - x_i\| \sum_{i=0}^{k} a_{i,k+1} + La_{k+1,k+1}\|x(t_{k+1}) - x_{k+1}^P\|\\
    = L \sup_{i=1,\ldots,k} \|x(t_i) - x_i\| \left(\frac{(\psi(t_{k+1}) - \psi(t_0))^{\alpha}}{\alpha}-\frac{(\psi(t_{k+1}) - \psi(t_k))^\alpha}{\alpha}\right)\\
    + L \left(\frac{(\psi(t_{k+1}) - \psi(t_k))^\alpha}{\alpha}\right) \|x(t_{k+1}) - x_{k+1}^P\|
\end{gather*}
Now consider,
\begin{gather*}
     \left\|\int_{t_0}^{t_{k+1}} (\psi(t_{k+1}) - \psi(\tau))^{\alpha-1} \psi'(\tau) g(\tau,x(\tau)) d\tau - \sum_{i=0}^{k+1} a_{i,k+1} g(t_i,x(t_i))\right\|\\
     = \left\|\int_{t_0}^{t_{k+1}} (\psi(t_{k+1}) - \psi(\tau))^{\alpha-1} \psi'(\tau) \left( g(\tau,x(\tau)) - \sum_{i=0}^{k+1} g(t_i,x(t_i))\phi_{i,k+1}(\tau)\right) d\tau\right\|\\
     \le h^2 \sup_{t} \left\| \frac{d^2}{dt^2} g(t,x(t)) \right\| \int_{t_0}^{t_{k+1}} (\psi(t_{k+1}) - \psi(\tau))^{\alpha-1} \psi'(\tau) d\tau\quad\text{(use \eqref{O(h2)})}\\
     = h^2 \sup_{t} \left\| \frac{d^2}{dt^2} g(t,x(t)) \right\|  \frac{(\psi(t_{k+1}) - \psi(t_0))^{\alpha}}{\alpha}.
\end{gather*}
The following inequalities have been established:
\begin{gather}
    \|x(t_{k+1}) - x_{k+1}\| \le h^2 \sup_{t} \left\| \frac{d^2}{dt^2} g(t,x(t)) \right\|  \frac{(\psi(t_{k+1}) - \psi(t_0))^{\alpha}}{\alpha} + \nonumber\\ \label{eqcorrector-inequalityyyyyy}
    L \sup_{i=1,\ldots,k} \|x(t_i) - x_i\| \left(\frac{(\psi(t_{k+1}) - \psi(t_0))^{\alpha}}{\alpha}-\frac{(\psi(t_{k+1}) - \psi(t_k))^\alpha}{\alpha}\right)\\
    + L \left(\frac{(\psi(t_{k+1}) - \psi(t_k))^\alpha}{\alpha}\right) \|x(t_{k+1}) - x_{k+1}^P\|\nonumber
\end{gather}
and
\begin{gather}\label{eq:predictor-inequality}
    \|x(t_{k+1}) - x_{k+1}^P\| \le 
 \frac{1}{\Gamma(\alpha)}\left(\sup_{t \in [t_0,t_{k+1}]} \left\| \frac{d}{dt} g(t,x(t)) \right\|
h  \frac{(\psi(t_{k+1}) - \psi(t_0))^\alpha}{\alpha}\right.\\
\left. + L \sup_{i=1,\ldots,k}\|x(t_i) - x_i\|  \frac{(\psi(t_{k+1}) - \psi(t_0))^\alpha}{\alpha} \right).\nonumber
\end{gather}

\begin{thm}\label{thm:convergence}
Let $\alpha\in (0,1]$. Suppose that the function $\psi:\R_{\geq a}\to\R_{\geq 0}$ satisfies Assumption \ref{asu-psi}(1). Let $x(t)$ be a solution to \eqref{FO-DE}, and suppose that $t \mapsto g(t,x(t))$ is twice continuously differentiable over $[t_0,T]$, and that $g$ is Lipschitz continuous. Then predictor-corrector scheme above yields the convergence rate \[ \|x(t_k) - x_{k}\| = O(h^{1+\alpha}) \] for a suitably chosen $T$.
\end{thm}

\begin{proof}
The proof will proceed by induction on $k$. For $k = 0$, $x_0 = x(t_0)$, so the result holds automatically. Now suppose that the result holds for $i=1,\ldots,k$. Then, by \eqref{eq:predictor-inequality}, 
\begin{gather*} \|x(t_{k+1}) - x_{k+1}^P\| \le \frac{1}{\Gamma(\alpha)} \left(\sup_{t \in [t_0,t_{k+1}]} \left\| \frac{d}{dt} g(t,x(t)) \right\|
h  \frac{(\psi(t_{k+1}) - \psi(t_0))^\alpha}{\alpha}\right.\\
\left.+ L \sup_{i=1,\ldots,k}\|x(t_i) - x_i\|  \frac{(\psi(t_{k+1}) - \psi(t_0))^\alpha}{\alpha} \right)\\
\le \frac{1}{\Gamma(\alpha)} \left(\sup_{t \in [t_0,t_{k+1}]} \left\| \frac{d}{dt} g(t,x(t)) \right\|
h  \frac{(\psi(t_{k+1}) - \psi(t_0))^\alpha}{\alpha} + L C h^{1+\alpha} \frac{(\psi(t_{k+1}) - \psi(t_0))^\alpha}{\alpha} \right).
\end{gather*}
To further simplify computations, note that \begin{gather*}
\frac{(\psi(t_{k+1})-\psi(t_0))^\alpha}{\alpha} \le \psi'(T)^{\alpha}\frac{(t_{k+1} - t_0)^\alpha}{\alpha} \le \psi'(T)^{\alpha}\frac{(T - t_0)^\alpha}{\alpha},\\
\text{ and } \frac{(\psi(t_{k+1})-\psi(t_k))^\alpha}{\alpha} \le \psi'(T)^{\alpha}\frac{h^\alpha}{\alpha}.
\end{gather*}
The applying the above to \eqref{eqcorrector-inequalityyyyyy} yields
\begin{align*} \|x(t_{k+1}) - x_{k+1}\| &\le h^2 \sup_{t} \left\| \frac{d^2}{dt^2} g(t,x(t)) \right\|  \psi'(T)^{\alpha}\frac{(T - t_0)^\alpha}{\alpha}\\
& + L C h^{1+\alpha} \left(\psi'(T)^{\alpha}\frac{(T - t_0)^\alpha}{\alpha} +\psi'(T)^{\alpha}\frac{h^\alpha}{\alpha}\right)\\
& +L \psi'(T)^{\alpha}\frac{h^{1+\alpha}}{\alpha\Gamma(\alpha)} \sup_{t \in [t_0,t_{k+1}]} \left\| \frac{d}{dt} g(t,x(t)) \right\| \psi'(T)^{\alpha}\frac{(T - t_0)^\alpha}{\alpha}\\
&+L^2\frac{h^{1+2\alpha}}{\alpha\Gamma(\alpha)} C \psi'(T)^{\alpha}\frac{(T - t_0)^\alpha}{\alpha}\\
&= h^{1+\alpha} \cdot \left( h^{1-\alpha} \sup_{t} \left\| \frac{d^2}{dt^2} g(t,x(t)) \right\|  \psi'(T)^{\alpha}\frac{(T - t_0)^\alpha}{\alpha}\right.\\
& + L C \left(\psi'(T)^{\alpha}\frac{(T - t_0)^\alpha}{\alpha} +\psi'(T)^{\alpha}\frac{h^\alpha}{\alpha}\right)\\
& +L \psi'(T)^{\alpha}\frac{1}{\alpha\Gamma(\alpha)} \sup_{t \in [t_0,t_{k+1}]} \left\| \frac{d}{dt} g(t,x(t)) \right\| \psi'(T)^{\alpha}\frac{(T - t_0)^\alpha}{\alpha}\\
&\left. +L^2 \frac{h^{\alpha}}{\alpha\Gamma(\alpha)} C \psi'(T)^{\alpha}\frac{(T - t_0)^\alpha}{\alpha}\right).
\end{align*}
With a selection of $T$ suitably close to $t_0$, the term in parentheses can be kept less than $C$. Hence, the convergence rate is established.
\end{proof}

\begin{rem}
Theorem \ref{thm:convergence} requires smoothness of both $g$ and $x$ for the convergence of the numerical method. Results for fractional order differential equations using the Caputo derivative show that $x$ inherits some smoothness on $(0,T]$ from $g$ (cf. \cite{diethelm2010analysis}), and these results are expected to carry over to the present case. Moreover, through the selection of $\psi$ having higher order zeros at $t=0$, numerical experiments indicate that $x$ may be smooth over $[0,T]$ as well.
\end{rem}


\subsection{Numerical examples}

Figure \ref{fig:validation} validates the numerical method through the comparison of the numerically generated solution to \eqref{eq:linear-equation} with the known closed form solution, $y(t) = E_{\alpha}(-(\psi(t) - \psi(0))^\alpha)$. The quantity $\alpha$ was selected to as $\alpha = 1/2$, and the figure presents the solutions to \eqref{eq:linear-equation} for $\psi(t) = t, t^2, t^3, t^4.$

Presented in Figure \ref{fig:booth}, Figure \ref{fig:exponential}, and Figure \ref{fig:zakharov} are the numerical experiments utilizing the functions given in \cite{liang2019fractional}. The gradient descent methods of this manuscript were utilized with $\psi = t, t*\ln(t+1), t^2, t^4$, and for comparison the integer order gradient descent using RK4 has also been included. It can be seen that setting $\psi =  t^2$ and $\psi =  t^4$ leads to better performance than integer order gradient descent.
 
The plots given show the norm convergence of the two dimensional state to the optimal point, so the ideal is to demonstrate convergence of this quantity to zero. The spacing $h$ is indicated in the title of each figure as is the fractional order of the derivative. For each fractional numerical method, the ABM method was used with $5$ corrector steps, where the corrector was applied iteratively to improve convergence.
 
On the Booth function, $f(x) = (x_1 + 2x_2 - 7)^2 + (2x_1 + x_2 - 5)^2$, it can be seen in Figure \ref{fig:booth} that with the initial condition $(10,5)$, $\alpha = 0.8$ and $\psi = t^4$, matching convergence rates are achieved between the integer order gradient descent method and the fractional gradient descent methods of this manuscript. Of the methods employed, the Caputo approached the optimal point most slowly. For the negative radial exponential function, $f(x) = -\exp\left(\frac{1}{2} \|x\|\right)$, with initial condition $(1,5)$, there is a much more dramatic difference between the methods, where $\psi =t^4$ converges very quickly as seen in Figure \ref{fig:exponential}. The noise after the trajectory achieves the optimal point is due to the nondifferentiability of the function at that point. Finally, for the Zakharov function, $f(x) = \|x\|_2^2 +  \left(\sum_{i=1}^n 0.5 i x_i \right)^2 + \left( \sum_{i=1}^n 0.5 i x_i\right)^4$ (with $n=2$ in this case), with initial condition $(10,5)$, integer order gradient descent performs poorly compared to $\psi = t^2$ and $\psi = t^4$, where the fastest convergence is achieved by $\psi = t^4$ in Figure \ref{fig:zakharov}.

Figure \ref{fig:compare_picard} compares the performance of the Picard method to that of the ABM method of the present manuscript in the numerical solution to the initial value problem with $g(t,x) = 1-2x-x^2$, $\psi(t) = t^4$, and $\alpha = 1/2$ where a closed form solution is unavailable. Here it is seen that the ABM method with $h=0.0001$ and $5$ corrector steps performs consistently throughout the interval $[0,1]$. However, the Picard method with $8$ iterates has a sudden spike. As there is no closed form solution for this IVP, this example lacks a ground truth for direct comparison. However, it should be noted that the Picard method breaks down after $8$ iterations and cannot improve upon this error. The instability validates the use of the ABM method in the other numerical experiments over that of the Picard method.

\begin{figure}
    \centering
    \includegraphics[scale=0.35]{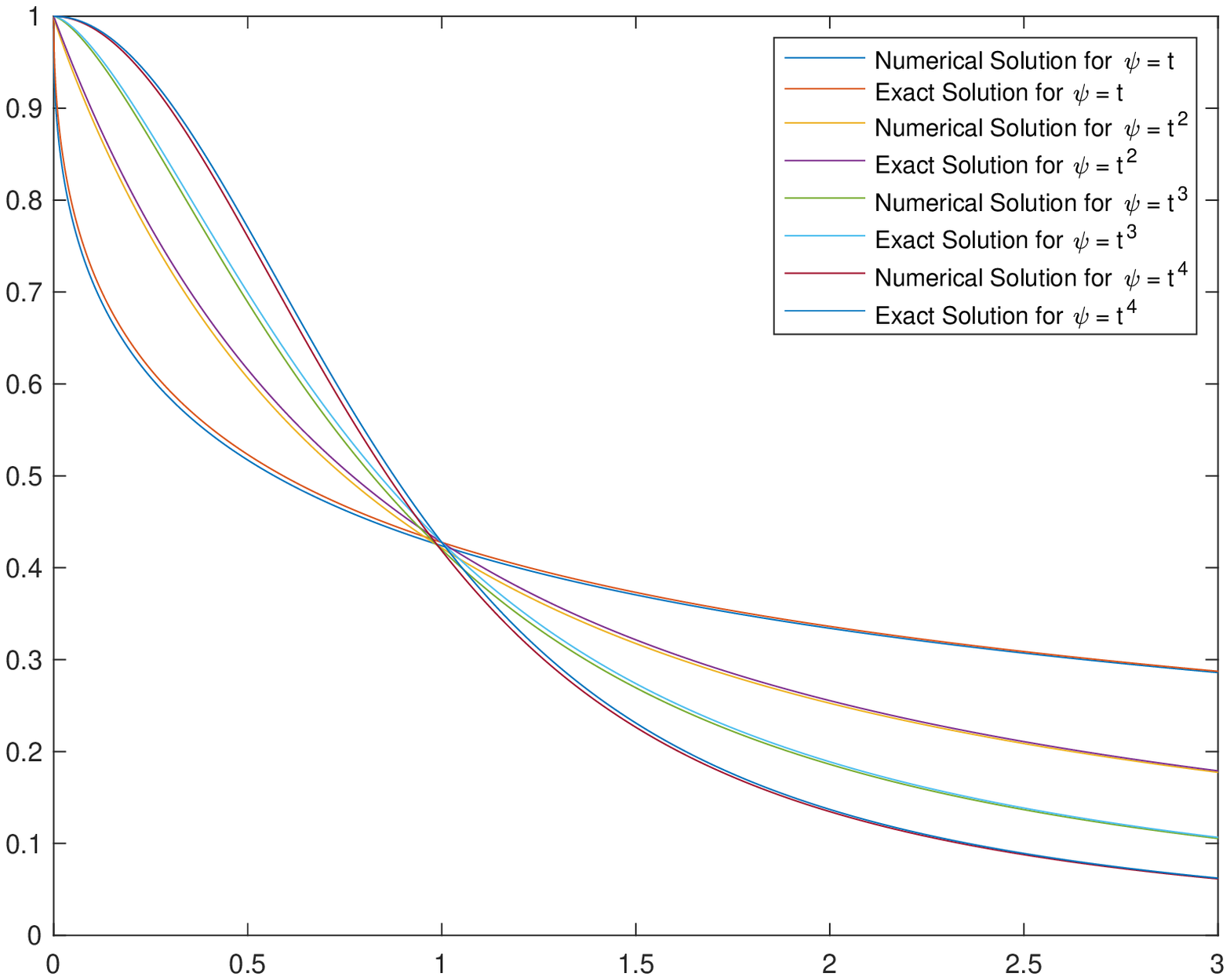}
    \includegraphics[scale=0.35]{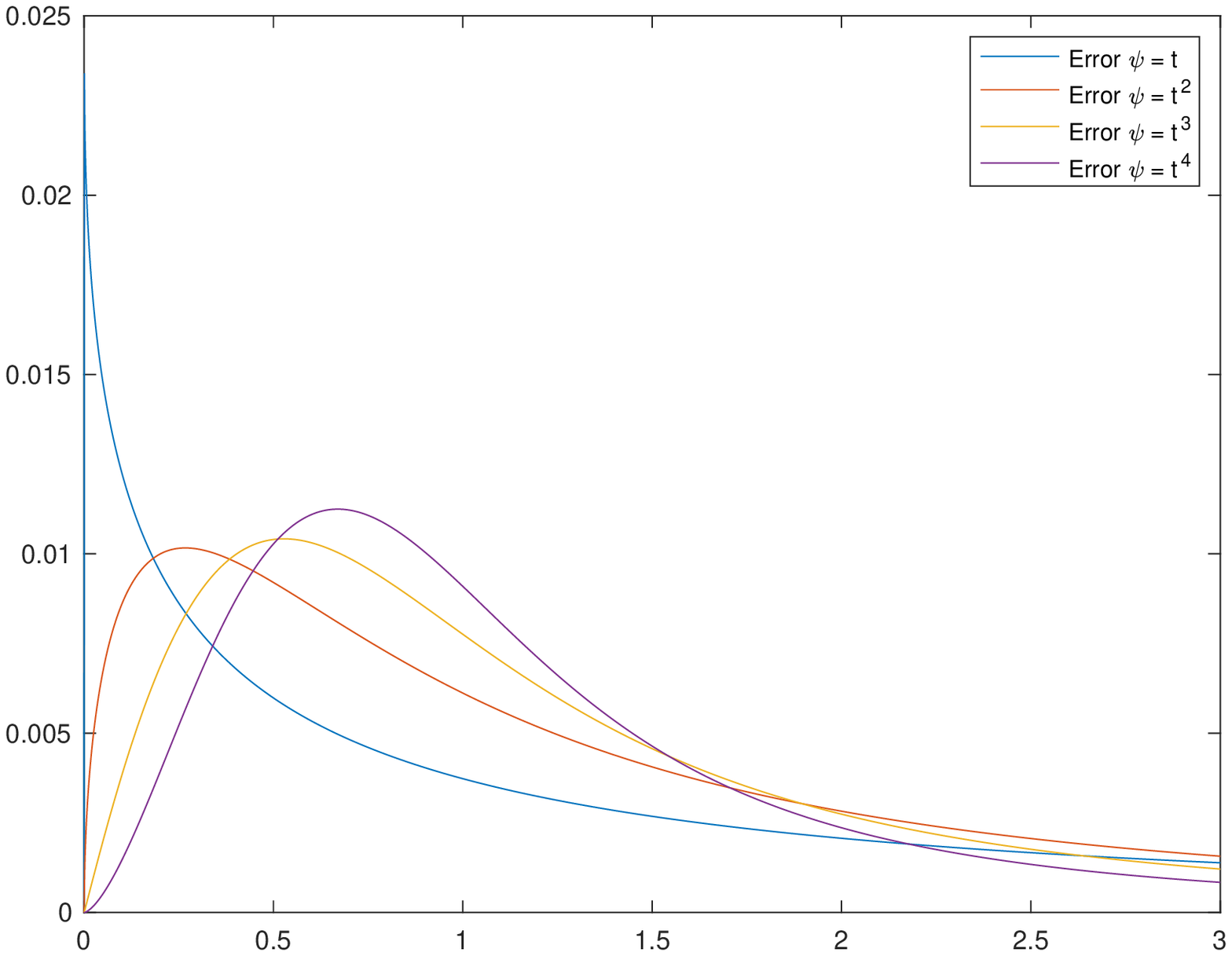}
    \caption{The figure on the left presents both the exact solutions to \eqref{eq:linear-equation} and the numerical solutions generated by the ABM method presented in this manuscript. Here $\alpha = 0.5$, $h = 0.001$, and $\psi(t) = t, t^2, t^3, t^4.$ The right figure presents plots of the respective absolute errors. The small error bounds validate the numerical method, where the selection of smaller $h$ yields smaller absolute errors (not presented).}
    \label{fig:validation}
\end{figure}

\begin{figure}
    \centering
    \includegraphics[scale=0.3]{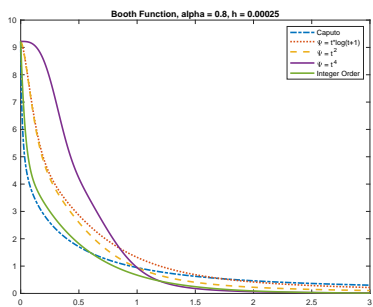}
    \caption{This figure shows the results of the gradient descent methods of this manuscript as applied to the optimization of the Booth function. This figure presents the norm difference between the known optimal point and the two dimensional state vector. Here it can be seen that the selection of $\psi = t^4$ leads to a method that matches integer order gradient descent.}
    \label{fig:booth}
\end{figure}

\begin{figure}
    \centering
    \includegraphics[scale=0.3]{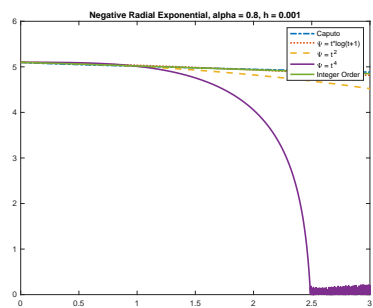}
    \caption{This figure shows the gradient descent methods of this manuscript as applied to the optimization of the negative radial exponential function. This figure presents the norm difference between the known optimal point and the two dimensional state vector.  Here it can be seen that the selection of $\psi = t^4$ gives a dramatic improvement over the other methods shown in this figure.}
    \label{fig:exponential}
\end{figure}

\begin{figure}
    \centering
    \includegraphics[scale=0.3]{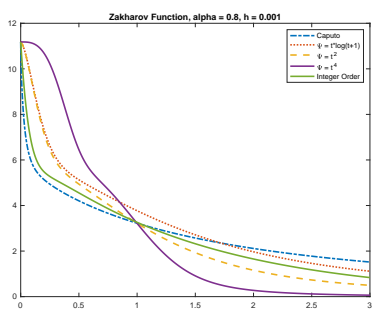}
    \caption{This figure shows the gradient descent methods of this manuscript as applied to the optimization of the Zakharov function. This figure presents the norm difference between the known optimal point and the two dimensional state vector.  Here it can be seen that the selection of $\psi = t^2$ and $\psi = t^4$ leads to a method that outperforms integer order gradient descent.}
    \label{fig:zakharov}
\end{figure}

\begin{figure}
    \centering
    \includegraphics[scale=0.3]{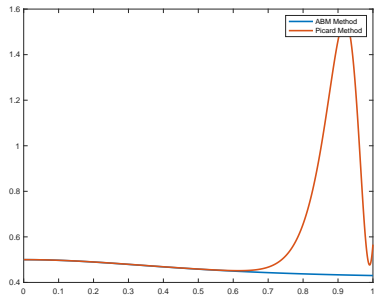}
    \caption{This figure compares the perfomance of the Picard iteration approach to solving the IVP for $g(t,x) = 1-2x-x^2$, $\psi(t) = t^4$, and $\alpha = 1/2$ with that of the ABM method developed in this manuscript. The Picard method has a sudden spike, and further iterations cause numerical errors. However, the ABM method performs consistently throughout the interval.}
    \label{fig:compare_picard}
\end{figure}
\section{Conclusions}\label{conl-sec}
In the paper, we design $\psi$-fractional derivatives based methods solving unconstrained optimization problems. The convergence analysis of these methods is carried out for both strongly convex and non-strongly convex cases. The key element of our analysis is the identification of a Lyapunov-type function, which allows to establish convergence of generated trajectories in the Riemann-Liouville as well as in the Caputo case. Chain rules and Jensen-type inequality play essential roles in the analysis of these Lyapunov functions. Numerical examples using the ABM method reveal that \emph{the fractional order $\alpha$ and weight $\psi$ are tunable parameters, which can be helpful for improving the convergence speed}. Future research may include extensions to constrained optimization problems, discrete-time methods and improvement of the convergence speed for specific problems. Moreover, the numerical methods given in Section \ref{de-num-me} generalize that of \cite{diethelm2010analysis}, and they have the same limitations where the convergence rate is valid only for a finite time horizon. Future developments for the numerical methods will be to provide a method and proof of convergence that is valid for arbitrarily large time horizons.

\section*{Acknowledgments}
We thank the Reviewers for their careful reading and insightful assessment of our work. Dr. Joel A. Rosenfeld was supported by the Air Force Office of Scientific Research (AFOSR) under contract number FA9550-20-1-0127 and the National Science Foundation under NSF Award ID 2027976.

\section*{CONFLICT OF INTEREST}
This work does not have any conflicts of interest.

\bibliographystyle{plain}
\bibliography{refs}
\end{document}